\documentclass[11pt]{article}
\usepackage{epsfig,amsmath,amssymb,amsfonts,amstext,amsthm,mathrsfs}
\usepackage{latexsym,graphics,epsf,epsfig,psfrag}
\usepackage{cite,dsfont,color,epstopdf,fixmath}

\newcommand{\bd}{\text{\boldmath{$d$}}}

\newcommand{\mg}{\mathrm \Gamma}
\newcommand{\ml}{\mathrm \Lambda}
\newcommand{\mo}{\mathrm\Omega}

\newtheorem{corollary}{\textbf{Corollary}}
\newtheorem{lemma}{\textbf{Lemma}}
\newtheorem{theorem}{\textbf{Theorem}}
\newtheorem{proposition}{\textbf{Proposition}}
\newtheorem{remark}{\textbf{Remark}}

\newtheorem{example}{\textbf{Example}}

\newcommand{\nn}{\nonumber}

\newcommand{\mE}{\mathbb{E}}
\newcommand{\mP}{\mathbb{P}}

\newcommand{\ess}{\text{ess}}
\usepackage{xcolor,ulem,pdfsync}
\newcommand{\gf}[1]{\textcolor{red}{#1}}

\newcommand\redsout{\bgroup\markoverwith{\textcolor{red}{\rule[0.5ex]{2pt}{0.8pt}}}\ULon}
\newcommand{\zsfd}[1]{\ifmmode\text{\redsout{\ensuremath{#1}}}\else\redsout{#1}\fi}

\begin{document}

\begin{center}
	\baselineskip 1.3ex {\Large \bf Quickest Change Detection under Transient Dynamics: Theory and Asymptotic Analysis
	\let\thefootnote\relax\footnote{The material in this paper was presented
		in part at the IEEE International Symposium on Information Theory (ISIT), Aachen, Germany, June
		2017~\cite{Zou2017QCD}.}
	\let\thefootnote\relax\footnote{The work of S. Zou and V. V. Veeravalli was supported in part by the National Science Foundation (NSF) under grants CCF 16-18658  and CIF 15-14245,  by the Air Force Office of Scientific Research (AFOSR) under grant FA9550-16-1-0077, and  by ARL under cooperative agreement W911NF-17-2-0196, through the University of Illinois at Urbana-Champaign.
		The work of G. Fellouris was supported by the NSF under grant CIF 15-14245, through the University of Illinois at Urbana-Champaign.}}
	\\
	\vspace{0.15in} Shaofeng Zou$^\dagger$, Georgios Fellouris*, Venugopal V. Veeravalli*
	\\
	\vspace{0.15in}$^\dagger$University at Buffalo, the State University of New york
	\\
	\vspace{0.15in}*University of Illinois at Urbana-Champaign
	\\
	\vspace{0.15in}Email: szou3@buffalo.edu, fellouri@illinois.edu, vvv@illinois.edu
	
\end{center}

	\begin{abstract}
		The problem of quickest change detection (QCD) under transient dynamics is studied, where the change from the initial distribution to the final persistent distribution does not happen instantaneously, but after a series of transient phases. The observations within the different phases are generated by different distributions. The objective is to detect the change as quickly as possible, while controlling the average run length (ARL) to false alarm, when the durations of the transient phases are completely unknown.
		Two algorithms are considered, the dynamic Cumulative Sum (CuSum) algorithm, proposed in  earlier work, and a newly constructed weighted dynamic CuSum algorithm.  Both algorithms admit recursions that facilitate their practical implementation, and they are adaptive to the unknown transient durations.
		Specifically, their asymptotic optimality is established with respect to both Lorden's and Pollak's criteria as the ARL to false alarm and the durations of the transient phases go to infinity at any relative rate.
		Numerical results are provided to demonstrate the adaptivity of the proposed algorithms, and to validate the theoretical results.
	\end{abstract}
	\section{Introduction}
	
	In the problem of quickest change detection (QCD), a decision maker obtains observations sequentially, and at some unknown time (\textit{change-point}), an event occurs and causes the distribution of the subsequent observations to undergo a change. The objective of the decision maker is to find a stopping rule that detects the change as quickly as possible, subject to a constraint on the false alarm rate. In classical QCD formulations \cite{veeravalli2013quickest,poor-hadj-qcd-book-2009,basseville1993detection,tartakovsky2014sequential,aminikhanghahi2017survey,truong2018review}, the statistical behavior of the samples is characterized by the pre-change distribution and \textit{one} post-change distribution, which generate the samples before and after the change-point, respectively. However, there are many practical applications with more involved statistical behavior after the change-point. For example, when a line outage occurs in a power system, the system goes through multiple transient phases before entering a persistent phase \cite{Rovatsos2:2016}. 
	
	Motivated by this type of applications, in this work we study the problem of QCD under transient post-change dynamics, in which the pre-change distribution does not change to the persistent distribution instantaneously, but after a number of transient phases. Within the transient and persistent phases, the observations are generated by distributions different from the initial one, and the problem is to detect the change as soon as possible either during a transient phase or during the persistent phase.  
	
	\subsection{Related Work}
We first stress that the  QCD problem under transient post-change dynamics that we study in this work is  different from the problem of detecting \textit{transient changes}, studied in \cite{guepie2012} and \cite{ebrahimzadeh2015sequential}, in which the system goes back to its pre-change mode after a single transient phase, and where it is only possible to detect the change within the transient phase.	Moreover, the setup in this paper  is fundamentally  different from the  model selection setup  in \cite{van2014almost,davis2013consistency}, where the minimum description length (MDL) principle is used to estimate the number of transient phases and the location of the change-points, \textit{with a fixed number of observations}. Here, we consider the case where the observations are collected \textit{sequentially}, and we are only interested in detecting 
in real time whether the distributions no longer follow the  pre-change distribution.

	
	
	The QCD problem under transient dynamics was studied in \cite{moustakides2016sequentially} when there is only \textit{one} transient phase that lasts for a \textit{single} observation, where a generalization of  Page's Cumulative Sum (CuSum) algorithm \cite{page1954continuous} is proposed  and shown to be optimal under Lorden's criterion \cite{lorden1971procedures}.
	A Bayesian formulation is proposed in \cite{george2017icassp}, in which it is assumed that there is an arbitrary, yet known, number of transient phases, whose durations are geometrically distributed. The proposed algorithm in \cite{george2017icassp} is a generalization of the Shiryaev-Roberts rule \cite{shiryaev1963optimum,shiryaev2007optimal}.  A non-Bayesian formulation is considered in \cite{Rovatsos2:2016}, where it is assumed that the durations are deterministic and \textit{completely unknown}. The proposed algorithm in \cite{Rovatsos2:2016} is a generalization of Page's CuSum test,  called the \textit{dynamic CuSum} (D-CuSum) algorithm. 	The algorithms in \cite{Rovatsos2:2016} and \cite{george2017icassp} are shown to admit a recursive structure, but are not supported by any theoretical performance analysis.

To be more precise,  the D-CuSum algorithm was derived in \cite{Rovatsos2:2016}  by reformulating the QCD problem as a dynamic composite hypothesis testing problem, where a hypothesis test  is conducted at each time instant $k$  until a stopping criterion is met. At each time $k$, the null hypothesis corresponds to the case that the change from the pre-change distribution has not occurred yet, and the alternative hypothesis corresponds to the case that the change has already occurred. Under the null hypothesis, all samples are distributed according to the pre-change distribution; under the alternative hypothesis, the distribution of the samples up to time $k$ depends on the unknown change-point and the unknown durations of the transient phases, and is thus composite. The test statistic at time $k$ is the \textit{generalized} likelihood ratio between the two hypotheses, and the corresponding stopping rule is obtained by comparing the test statistic against a pre-specified threshold.  We stress that the implementation of this algorithm does not require knowledge of the transient durations, which are considered to be deterministic and completely unknown.

Since the post-change distribution in our formulation is determined by the unknown durations of the transient phases,  the proposed problem falls into the framework of QCD with a composite post-change distribution \cite{lai1998information,lai1995sequential,siegmund1995using}.
Our work differs from this literature  in three major ways. First, thanks to the special structure of our problem, the proposed detection statistics enjoy recursions, a very important feature  for practical implementation that is  typically absent in \cite{lai1998information,lai1995sequential,siegmund1995using}.  Second, our asymptotic analysis is  novel and challenging due to the fact that  it requires not only the ARL to false alarm, but also the parameters of the post-change distribution (transient durations) to go to infinity. Third, the distribution of the samples within each phase can be arbitrary (may not belong to an exponential family), and the parameters of the post-change distribution (transient durations) are discrete and do not belong to a compact parameter space.

\subsection{Main Contributions}
The first contribution of this work is  that, under a certain condition on the pre/post-change distributions, we obtain a lower bound on the ARL to false alarm of the D-CuSum algorithm, which can be used for an explicit  selection of the threshold.  

The second contribution of this paper is that we propose an alternative algorithm, to which we refer as \textit{weighted} D-CuSum (WD-CuSum) algorithm, which also admits a recursive structure and  for which  we derive a \textit{universal} lower bound on the ARL.  This algorithm is a modification of the D-CuSum algorithm, where   the test statistic at time $k$ is a \textit{weighted} generalized likelihood ratio between the two hypotheses described above. The key idea is that instead of taking a maximum likelihood approach with respect to the  composite alternative hypothesis as in the D-CuSum algorithm, we take a mixture approach and then replace the sum in the mixture with a max in order to obtain a recursive structure for the resulting algorithm. 

The third contribution of this work is that we conduct an asymptotic analysis for the  performance of   D-CuSum  and  WD-CUSUM, which  demonstrates the statistical efficiency and adaptivity of both algorithms to the unknown transient durations. For this analysis, we adopt a worst-case scenario for the unknown change-point, and the  performance metrics of interest that we consider are the worse-case average detection delays (WADD) as defined by Lorden  \cite{lorden1971procedures} and Pollak \cite{pollak1985optimal}. As mentioned earlier, the exact minimizer of Lorden's WADD has been derived only in the special case of a single transient phase that lasts for a single observation \cite{moustakides2016sequentially}. As the proposed algorithms \textit{do not make any assumptions regarding the (deterministic) durations of the transient phases},  our goal is to show that they have ``good'' performance under any possible transient durations. In order to  do so,  we obtain an asymptotic  lower bound of the WADD  as the durations of the transient phases  and the ARL go to infinity \textit{at any relative rate}, and we show that this  lower bound 
 is achieved by the WD-CuSum algorithm. We stress that this asymptotic optimality is achieved for \textit{any}  divergence rate of  the transient phases, implying the adaptivity of the algorithm to the unknown transient durations. Similar asymptotic optimality  results are also obtained for the D-CuSum algorithm.

In order to demonstrate the performance of the proposed algorithms,  we conduct a simulation study which illustrates  how  both the D-CuSum and the WD-CuSum algorithms  are adaptive to the unknown transient durations and have similar performances for any practical purposes.
Moreover, we propose a heuristic approach for the selection of the
weights of the WD-CuSum algorithm in the finite regime, which balances the performance within the transient and persistent phases.

	\subsection{Paper Organization}
	
	The remainder of this paper is organized as follows. In Section \ref{sec:model}, we formulate the problem mathematically.
	In Section \ref{sec:algorithms}, we introduce the D-CuSum and the WD-CuSum algorithms. In Section \ref{sec:lowerarl} we establish lower bounds on the ARL to false alarm for both algorithms. In Section \ref{sec:asymptotic}, we demonstrate the asymptotic optimality of both algorithms.  In Section \ref{sec:numerical}, we present the numerical results and propose a heuristic approach of choosing weights for the WD-CuSum algorithm. Finally, in Section \ref{sec:con}, we provide some concluding remarks.
	
	\section{Problem Model}\label{sec:model}

	Consider a sequence of independent random variables $\{X_k\}_{k=1}^\infty$, observed sequentially by a decision maker. At an unknown change-point $v_1$, an event occurs and
	$\{X_k\}_{k=v_1}^\infty$ undergoes a change in distribution from the initial distribution, $f_0$. It is assumed that this change goes through $L-1$ transient phases before entering a persistent phase.  Each phase $i$ begins at an unknown starting point $v_i$, and the observations within this phase are generated by a known distribution $f_i$, for $1\leq i\leq L$. The duration of $i$-th transient phase  is denoted by $d_i= v_{i+1}-v_i$, for $1\leq i\leq L-1$. More specifically, the observations are distributed as follows:
	\begin{flalign}
	X_k\sim f_i, \text{ if } v_i\leq k<v_{i+1},
	\end{flalign}
	for $0\leq i\leq L$, where $v_0= 1$, $v_1\leq v_2\leq\cdots\leq v_L$, and $v_{L+1}=\infty$. We assume that $L$ is known in advance and so are the densities $f_i$, $0\leq i\leq L$.  	On the other hand, the change point $v_1$ and the vector of transient durations $\bd= \{d_i,1\leq i\leq L-1\}$ are assumed to be deterministic and completely \textit{unknown}.

	The goal is to detect the change reliably and quickly based on the sequentially acquired observations. That is, if $\mathcal F_k$ is the $\sigma$-algebra generated by the first $k$ observations, i.e., $\mathcal F_k=\sigma(X_1,\ldots,X_k)$, where $k=1,2,\ldots$, we want to find a $\{\mathcal F_k\}_{k\in \mathbb N}$-stopping time that achieves ``small" detection delay, while controlling the rate of false alarms.
	
To be more specific,  we denote by $\mP_{\infty}$ and $\mE_{\infty}$ the probability measure and the corresponding expectation when $v_1=\infty$, i.e., when there is no change, and  for any stopping time $\tau$ we define the ARL to false alarm as follows:
\begin{flalign}
	\text{ARL}(\tau)=&\mE_{\infty}[\tau].
\end{flalign}
The first requirement for a stopping rule is to  control the expected time to false alarm above a user-specified level, $\gamma>1$, i.e., to belong to  $\mathcal C_\gamma=\{\tau:  \text{ARL}(\tau)\geq \gamma\}$.

In order to quantify the performance of a stopping rule, we 
 denote by   $\mP^{\bd}_{v_1}$   the probability measure with the change-point at $v_1$ and the vector of transient durations $\bd$, and we denote by $\mE^\bd_{v_1}$ the corresponding expectation. Then, for a given  vector of transient durations $\bd$, the worst-case average detection delay of  a stopping time $\tau$ under Pollak's criterion \cite{pollak1985optimal} is  
	\begin{flalign}
	J^\bd_\text{P}(\tau)=&\sup_{v_1 \geq 1}\mE^\bd_{v_1}[\tau-v_1|\tau \geq v_1],
	\end{flalign}
and  under Lorden's criterion \cite{lorden1971procedures} 
	\begin{flalign}
	J^\bd_\text{L}(\tau)=&\sup_{v_1\geq 1}\ess\sup\mE^\bd_{v_1}[(\tau-v_1)^+|X_1,\ldots,X_{ v_1-1}],
	\end{flalign}
	where $(\tau-v_1)^+=\max\{\tau-v_1,0\}$. Thus, for any  given vector of transient durations $\bd$, we have the following two optimization problems:
	\begin{flalign}
	\underset{\tau\in\mathcal C_\gamma}{\text{inf}} J^\bd_\text{P}(\tau), \label{eq:p} \\
	\underset{\tau\in\mathcal C_\gamma}{\text{inf}} J^\bd_\text{L}(\tau). \label{eq:l}
	\end{flalign}
An exact solution to \eqref{eq:l} has been obtained   for any given $\gamma$  only when $L=1$ and $d_1=1$ \cite{moustakides2016sequentially}. To the best of our knowledge, problem  \eqref{eq:p} has not been solved for any value of $L$ and $\bd$.  However, our interest in this work is on the case that the  vector of transient durations $\bd$ is completely unknown,  thus, our goal is not on solving \eqref{eq:l} or \eqref{eq:p} for a particular choice of $\bd$. Instead, our goal is to obtain algorithms that (i) control the ARL to false alarm and (ii) have a small WADD \textit{for any value of $\bd$}.  Specifically, we will introduce two stopping rules  (Section \ref{sec:algorithms}),  show how to design them in order to belong to $C_\gamma$ for any user-specified $\gamma>1$ (Section \ref{sec:lowerarl}), and also show that they  attain \eqref{eq:p} and \eqref{eq:l} up to a first-order approximation as $\gamma \rightarrow \infty$ and $d \rightarrow \infty$ (Section \ref{sec:asymptotic}).
	
	
	

\subsection{Notation}
For $i=1,\ldots, L$, we denote by
$$I_i=\int f_i \, \log \frac{df_i}{df_0}$$ the Kullback-Leibler (KL) divergence between $f_i$ and $f_0$,  which we assume to be positive and finite.  
For $i=1,\ldots, L$,  we set 
	\begin{flalign}
	Z_i(X_k)=\log\frac{f_i(X_k)}{f_0(X_k)},
	\end{flalign}
	i.e., $Z_i(X_k)$ the log-likelihood ratio  between $f_i$ and $f_0$ for  sample $X_k$, $i=1,\ldots, L$, $k=1,2, \ldots$.	Moreover, we set
	\begin{flalign}
	 \ml_i[k_1,k_2]= \prod_{j=k_1}^{k_2}  \frac{f_i(X_j)}{f_0(X_j)}, \; \ml_i[k_1,k_2)= \prod_{j=k_1}^{k_2-1}  \frac{f_i(X_j)}{f_0(X_j)}.
	\end{flalign}
	
	We denote the largest integer that is smaller than $x$ as $\lfloor x\rfloor$, and the smallest integer that is larger than $x$ as $\lceil x\rceil$. We define $\sum_{j=n_1}^{n_2}X_j=0$ and $\prod_{j=n_1}^{n_2}X_j=1$ if $n_1>n_2$. We denote $x=o(1)$, as $c\rightarrow c_0$ if $\forall\epsilon>0$, $\exists\delta>0$, s.t., $|x|\leq\epsilon$ if $|c-c_0|<\delta$. We denote $g(c)\sim h(c)$ as $c\rightarrow c_0$, if $\lim_{c\rightarrow c_0} \frac{f(c)}{g(c)}=1$.

\section{The Algorithms}\label{sec:algorithms}

In this section, we introduce the proposed algorithms,  we show that they admit simple recursive structures. In order to do so, we reformulate the QCD problem as a dynamic composite hypothesis testing problem, as in \cite{Rovatsos2:2016}, where which  at each time instant $k$ we  distinguish the following two hypotheses
\begin{flalign}\label{eq:hypothesistesting}
\mathcal H_0^k&:k<v_1,\nn\\
\mathcal H_1^k&:k\geq v_1.
\end{flalign}
The  process stops once a decision in favor of the alternative hypothesis is reached; otherwise, a new sample is taken. Under  $\mathcal H_0^k$, the samples $X_1,\ldots, X_k$ are distributed according to $f_0$. The alternative hypothesis $\mathcal H_1^k$ is composite, since it depends on  $v_1, \bd$, which are unknown.

Let  $\mg(k, v_1, \bd)$ denote the likelihood ratio of the first $k$ observations, $X_1,\ldots, X_k$, for fixed $v_1,\bd$, i.e.,
\begin{flalign}\label{eq:gamma}
\mg(k, v_1, \bd) = \frac{\mP_{v_1}^\bd (X_1,\ldots,X_k)}{\mP_\infty(X_1,\ldots,X_k)} .
\end{flalign}
When $v_i\leq k< v_{i+1}$ for some $1\leq i\leq L$,
\begin{flalign}
\mg(k,v_1,\bd)=\ml_i[v_i,k]\cdot\prod_{j=1}^{i-1}\ml_j[v_j,v_{j+1}).
\end{flalign}
For the special case with $L=2$,
\begin{flalign}
&\mg(k, v_1, d_1) \nn\\
&=
\begin{cases}
\ml_{1}[v_1,k], &\text{ if }  v_1+d_1 > k  ,\\
\ml_{1}[v_1,v_1+d_1)
\ml_{2}[v_1+d_1,k],
&\text{ if } v_1+ d_1 \leq  k.
\end{cases}
\end{flalign}
We note that $\nu_1$ can be equal to $\nu_2$, i.e., $d_1$ can be equal to 0. This implies that for a given pair of $(v_1,k)$, $k\geq\nu_1$, depending on the time that the second change takes place, there are $k-v_1+2$ possible values of $\mg(k,v_1,d_1)$. 
In general, when $L\geq 2$, for a given pair of $(k,v_1)$, $k\geq \nu_1$, depending on the time that the second to the $L$-th changes take place, there are also finitely many possible values of $\mg(k,v_1,\bd)$.

%


\subsection{D-CuSum}

The D-CuSum \cite{Rovatsos2:2016} detection statistic at time $k$ is  the generalized log-likelihood ratio with respect to both $v_1$ and $\bd$, for the above hypothesis testing problem:
\begin{flalign}\label{eq:dcusum_test}
\widehat W[k]&=\max_{1\leq v_1\leq k}\max_{\bd\in \mathbb N^{L-1}}\log   \mg(k, v_1, \bd) .
\end{flalign}
 
As we explained above, there are finitely many subhypotheses under $\mathcal H_1^k$, which implies that $\mg(k, v_1, \bd)$ has finitely many values, and the maximization in \eqref{eq:dcusum_test} is over finitely many terms. More specifically, equation \eqref{eq:dcusum_test} is equivalent to the one which takes maximization over
\begin{flalign}
\{(v_1,\ldots,v_L):1\leq v_1\leq k, v_1\leq v_2\leq \cdots\leq v_L\leq k+1\},
\end{flalign}
in which, each tuple of $(v_1,\ldots,v_L)$ corresponds to a distinct value of $\mg(k, v_1, \bd)$. 


The corresponding stopping time is given by comparing $\widehat W[k]$ against a pre-determined positive threshold:
\begin{flalign}\label{eq:dcusum}
\widehat \tau(b)=\inf\{k\geq1: \widehat W[k]>b\}.
\end{flalign}
Since $b>0$, without loss of generality we can adopt the positive part of $\widehat W[k]$ as the detection statistic.
It can be shown that
\begin{flalign}
&(\widehat W[k])^+\nn\\
&=\max_{1\leq v_1\leq \cdots\leq v_L\leq k+1} \log\frac{\prod_{i=1}^{L}\left(\prod_{j=v_i}^{\min\{v_{i+1}-1,k\}}f_i(X_j)\right)}{\prod_{j=v_1}^{k}f_0(X_j)}\nn\\
&=\max_{1\leq v_1\leq \cdots\leq v_L\leq k+1}\hspace{-0.05\linewidth}\sum_{j=v_1}^{\min\{v_2-1,k\}}Z_1(X_j)  +\cdots + \sum_{j=v_L}^{k}Z_L(X_j)     .
\end{flalign}

It is shown in \cite[Appendix]{Rovatsos2:2016} that $(\widehat W[k])^+$ has a recursive structure:
\begin{flalign}
(\widehat W[k])^+=\max\left\{ \widehat \mo^{(1)}[k],\widehat \mo^{(2)}[k],\ldots,\widehat \mo^{(L)}[k],0\right\},
\end{flalign}
where for $1\leq i\leq L$, we set $\widehat \mo^{(i)}[0]=0$ and
\begin{flalign}\label{eq:recursion_dcusum}
\widehat \mo^{(i)}[k]=\max\left\{  0,\widehat \mo^{(1)}[k-1],\ldots,\widehat \mo^{(i)}[k-1]   \right\}+Z_i(X_k).
\end{flalign}
\begin{remark}\label{remark1}
	The $L$-dimensional random  vector  $\{ \widehat \mo^{(1)}[k],\ldots,\widehat \mo^{(L)}[k]\}$  depends on $X_1,\ldots,X_{k-1}$ only through $\{\widehat \mo^{(1)}[k-1],\ldots,\widehat \mo^{(L)}[k-1]\}$; thus, it is a  Markov process.
		When all of its components are simultaneously non-positive, or equivalently when $(\widehat W[k])^+$ equals $0$ at some $k$, then $(\widehat W[k])^+$ forgets all previous observations and restarts from zero, i.e., it regenerates.
	
\end{remark}

\subsection{WD-CuSum}
If we take a mixture approach with respect to $\bd$, combined with  a maximum likelihood approach with respect to $v_1$, this suggests the following stopping rule:
\begin{flalign}
\tau'(b)=\inf\{k\geq 1: W'[k]\geq b\},
\end{flalign}
where $b$ is a positive threshold and the detection statistic is
\begin{flalign}\label{eq:38}
W'[k]=\max_{1\leq v_1\leq k}\log\left(\sum_{\bd\in \mathbb N^{L-1}}\mg(k,v_1,\bd)g(\bd)\right),
\end{flalign}
and $g$ is a pmf on $\mathbb N^{L-1}$. Recall that for a given pair of $(k,\nu_1)$, $k\geq \nu_1$, there are finitely many possible values of $\mg(k,v_1,\bd)$. Therefore, this mixture in \eqref{eq:38} is equivalent to a sum over finitely many terms. Here, we denote the total number of all possible values of $\mg(k,v_1,\bd)$ by $n(k,\nu_1)$, and denote each distinct value of $\mg(k,v_1,\bd)$ by $\lambda(k,v_1,j), 1\leq j\leq n(k,\nu_1)$. Then, 
%
\begin{flalign}
W'[k]=\max_{1\leq v_1\leq k}\log\left(\sum_{j=1}^{n(k,v_1)}\lambda(k,v_1,j)g_j\right),
\end{flalign}
where 
\begin{flalign}
g_j=\sum_{\bd:\mg(k,v_1,\bd)=\lambda(k,v_1,j)}g(\bd).
\end{flalign}
Replacing the sum with a maximum, we obtain 
\begin{flalign}
\widetilde W[k]=\max_{1\leq v_1\leq k}\hspace{0.1cm}\max_{1\leq j\leq n(k,v_1)}\log\big(\lambda(k,v_1,j)g_j\big),
\end{flalign}
which leads to the following stopping rule:
\begin{flalign}\label{eq:wdcusum}
\widetilde \tau (b)=\inf\{k\geq 1: \widetilde W[k]\geq b\}.
\end{flalign}
We refer to this stopping rule in \eqref{eq:wdcusum} as the WD-CuSum algorithm.
We note that the reason we replace the sum with a maximum is that with a particular choice of $g$, the resulting algorithm has a recursive structure, which can be updated efficiently.

In the following, we focus on $\widetilde{\tau}$ for a particular choice of $g$, which yields a recursive structure for $\widetilde{W}$. In particular, if we choose
\begin{flalign}
g(\bd)=\prod_{i=1}^{L-1}\rho_i(1-\rho_i)^{d_i},
\end{flalign}
for some $\rho_i\in(0,1)$, $1\leq i\leq L-1$, and consider the positive part of $\widetilde{W}[k]$ (since $b>0$), then
\begin{flalign}\label{eq:wdcusum_test}
(\widetilde W[k])^+=\max_{1\leq v_1\leq \cdots\leq v_L\leq k+1}\log\left(\frac{    \prod_{i=1}^LB_i }{\prod_{j=v_1}^{k}f_0(X_j)}\right),
\end{flalign}
where for $i=1,\ldots,L$,
\begin{flalign}
B_i{=}\left(\prod_{j=v_i}^{\min\{v_{i+1}-1,k\}}f_i(X_j)(1-\rho_{i}) \right) \rho_i^{\mathds{1}_{\{k\geq v_{i+1}\}}},
\end{flalign}
with $v_{L+1}=\infty$ and $\rho_L=0$.

Following steps similar to those in \cite[Appendix]{Rovatsos2:2016}, it can be shown that
\begin{flalign}\label{eq:recursion_wdcusum}
(\widetilde W[k])^+=\max\left\{  \widetilde \mo^{(1)}[k],\ldots,\widetilde \mo^{(L)}[k] ,0 \right\},
\end{flalign}
where  for     $1\leq i\leq L$,
\begin{flalign}\label{eq:37}
\widetilde \mo^{(i)}[k]=&\max_{0\leq j\leq i} \bigg(\widetilde \mo^{(j)} [k-1]+\sum_{\ell=j}^{i-1}\log\rho_{\ell}\bigg)  \nn\\
&+ Z_i(X_k) +\log(1-\rho_i), 
\end{flalign}
with $\widetilde \mo^{(0)}[k]=0$, for all $k$, and $\rho_0=1$.

\begin{example}		
	When  $L=2$, setting  $G(x)=\sum_{k>x}g(k)$, we have 
			\begin{flalign}\label{eq:long1}
		W'[k]&=\max_{1\leq v_1\leq k} \log\Bigg\{
		\sum_{d_1=0}^{k-\nu_1}  g(d_1) \ml_1[v_1,v_2)  \ml_2[v_2,k] \nn\\
		&\hspace{0.4\linewidth}+  G(k-\nu_1)  \ml_1[v_1,k]\Bigg\} ,\nn\\
		\widetilde W[k]&=\max_{1\leq v_1\leq k} \log\Bigg\{\max\bigg\{
		\max_{ 0 \leq d_1 \leq k-\nu_1}  g(d_1) \ml_1[v_1,v_2) \nn\\
		&\hspace{0.15\linewidth}\times \ml_2[v_2,k],   G(k-\nu_1)  \ml_1[v_1,k] \bigg\}\Bigg\}.
		\end{flalign}

\end{example}

\section{Lower Bounds on the  ARL}\label{sec:lowerarl}
In this section, we obtain non-asymptotic lower bounds on the ARL to false alarm for the D-CuSum algorithm and the WD-CuSum algorithm.

\subsection{D-CuSum}
It is interesting to point out that unlike the classical CuSum statistic, which we recover by setting $L=1$, $\{ \widehat \mo^{(1)}[k],\ldots,\widehat \mo^{(L)}[k]\}$ does not always regenerate under $\mP_{\infty}$ for $L\geq 2$. Denote by $Y$ the first regeneration time, i.e.,
\begin{flalign}\label{Y}
Y=\inf\{k \geq1: (\widehat{W}[k])^+ =  0\}.
\end{flalign}
The following example shows that $Y$ is not always finite.
\begin{example} \label{counter}
	Suppose that  $L=2$ and $f_0, f_1,f_2$ are chosen such that
	\begin{flalign}
	&f_0(x)=0.5\times\mathds 1_{\{x\in[0,2]\}}, \nn\\
	&f_1(x)=0.8\times \mathds 1_{\{x\in[0,1]\}} +0.2\times \mathds 1_{\{x\in(1,2]\}}, \nn\\
	& f_2(x)=0.2\times \mathds 1_{\{x\in[0,1]\}} +0.8\times \mathds 1_{\{x\in(1,2]\}}.
	\end{flalign}
	Then, 
	$
	\max\left\{ Z_1(x),Z_2(x)  \right\} > 0,
	$ 
	$\forall x\in[0,2]$, which implies that for all $ k\geq 1$, we have pathwise
	\begin{flalign}
	(\widehat{W}[k])^+=\max\left\{ \widehat{\mo}^{(1)}[k], \widehat{\mo}^{(2)}[k],0\right\}>0.
	\end{flalign}
\end{example}
If we assume that the pre- and post-change distributions satisfy the following condition:
\begin{flalign}\label{eq:conditions}
\mP_\infty(Y>m)\leq e^{-\alpha m}, \quad \forall   m\geq 1,
\end{flalign}
where $\alpha>0$ is a constant, then $Y$ is finite with probability one. In other words, the probability that $(\widehat{W}[k])^+$ regenerates within finite time is one.  Moreover, the expectation of the regeneration time $Y$ is upper bounded by a finite constant:
\begin{flalign}
\mE_\infty[Y]\leq 1+\sum_{m=1}^\infty e^{-\alpha m}\leq 1+\frac{1}{\alpha} < \infty.
\end{flalign}
Therefore, $(\widehat W[k])^+$ is regenerative, and the ARL of the D-CuSum algorithm is lower bounded as in the following proposition. See Appendix \ref{remark2} for an example of sufficient conditions for \eqref{eq:conditions} to hold.

\begin{proposition}\label{prop:lowerarldcusum}
	Consider the QCD problem under transient dynamics described in Section \ref{sec:model}.
	Assume that the pre- and post-change distributions satisfy condition \eqref{eq:conditions}. If the D-CuSum algorithm is applied with a threshold $b$, then the ARL is lower bounded as follows:
	\begin{flalign}
	\mE_{\infty}[\widehat \tau(b)]\geq \frac{e^b}{1+\left( {b/\alpha} \right)^{L+1}}.
	\end{flalign}
\end{proposition}
\begin{proof}
	See Appendix \ref{proof:proposition1}.
\end{proof}
\begin{corollary}
	Assume that the pre- and post-change distributions satisfy condition \eqref{eq:conditions}.
	To guarantee $\mE_{\infty}[\widehat \tau(b)]\geq\gamma$, it suffices to choose $b$
	such that 
	{
	$$
	\frac{e^{b}}{\left(b/\alpha\right)^{L+1}+1} = \gamma,
	$$
	}
	and 	$
	b\sim\log\gamma$.
\end{corollary}
\begin{proof}
	The result follows from Proposition \ref{prop:lowerarldcusum}.
\end{proof}

\subsection{WD-CuSum}

\begin{theorem}\label{thm:lowerboundarl}
	Consider the QCD problem under transient dynamics described in Section \ref{sec:model}. Assume that the WD-CuSum algorithm in \eqref{eq:wdcusum} is applied with threshold $b$ and any $\rho_i\in(0,1)$, $1\leq i\leq L-1$. Then, the ARL of  the WD-CuSum algorithm  is lower bounded as follows:
	\begin{flalign}
	\mE_\infty[\widetilde\tau(b)]\geq \frac{1}{2}e^b.
	\end{flalign}
\end{theorem}

\begin{proof}
	See Appendix \ref{proof:theorem1}.
\end{proof}
\begin{remark}
	The lower bound can be further tightened to $e^b$ by using
	Doob's optional sampling theorem \cite{robbins1971great}  instead of the submartingale inequality. However, this does not provide order-level improvement. 
\end{remark}
\begin{corollary}
	To guarantee $\mE_{\infty}[\widetilde \tau(b)] \geq\gamma$, it suffices to choose
	\begin{flalign}
	b=\log\gamma+\log 2\sim\log\gamma.
	\end{flalign}
\end{corollary}
\begin{proof}
	The result follows from Theorem \ref{thm:lowerboundarl}.
\end{proof}

\section{Asymptotic Analysis}\label{sec:asymptotic}
In this section  we study the asymptotic performance of the proposed algorithms and demonstrate their asymptotic optimality.  For our asymptotic analysis to be non-trivial, we let not only the prescribed lower bound on the ARL, $\gamma$,   go to infinity,  but also the transient durations. Indeed, if  the latter are  fixed as $\gamma$ goes to infinity, then
the CuSum algorithm that detects the change from $f_0$ to $f_L$, completely ignoring the  transient phases, can be shown to be
asymptotically optimal using the techniques in \cite{lai1998information}. Therefore, in order to perform a general and relevant  asymptotic analysis, we let $d_1, \ldots, d_{L-1}$ go to infinity with $\gamma$.
Without loss of generality, we assume that  there are constants $c_i \in[0,\infty]$,  $i=1,\ldots,L-1$ so that 
\begin{flalign} \label{scaling}
d_i\sim c_i  \frac{\log\gamma}{I_i},
\end{flalign}
where  $c_L=\infty$.  Note that  if  $d_i=o (\log\gamma)$, then $c_i=0$, whereas if $\log\gamma=o(d_i)$, then  $c_i=\infty$. We stress that $\{c_i, i=1,\ldots,L-1\}$ are unknown, since the transient durations are unknown, and are not utilized in the design of the proposed rules. However the optimal asymptotic performance will turn out to be a function of these constants.

We start with the case of a single transient phase $(L=2)$, since it captures the essential features of the analysis, and then present the generalization to $L>2$.


\subsection{Asymptotic Universal Lower Bound on the WADD}\label{sec:lower}
Consider the case with $L=2$, for which $\bd=d_1$. As will be shown in the following, the optimal asymptotic performance  depends on whether $c_1 \geq 1$ or $c_1<1$. This dichotomy can be seen in the following asymptotic universal  lower bound on the WADD.

\begin{theorem}\label{thm:lower}
	Consider the QCD problem under transient dynamics described in Section \ref{sec:model} with $L=2$. Suppose that \eqref{scaling} holds, i.e., $d_1\sim c_1\log\gamma/I_1$.
	\begin{enumerate}
		\item[(i)]  If $c_1\geq1$, then as $\gamma \rightarrow \infty$,
		\begin{flalign}\label{eq:11}
			\inf_{\tau\in\mathcal C_\gamma} J^{d_1}_\text{L}(\tau)\nn &\geq \inf_{\tau\in\mathcal C_\gamma}J^{d_1}_\text{P}(\tau) \nn\\
			&\geq \frac{\log \gamma}{I_1}(1-o(1));
		\end{flalign}
		\item[(ii)]  if $c_1<1$, then as $\gamma \rightarrow \infty$,
		\begin{flalign}\label{eq:22}
			\inf_{\tau\in\mathcal C_\gamma} J^{d_1}_\text{L}(\tau)\nn &\geq \inf_{\tau\in\mathcal C_\gamma}J^{d_1}_\text{P}(\tau) \nn\\
			&\geq \log \gamma \left(\frac{1-c_1}{I_2}+\frac{c_1}{I_1}\right)(1-o(1)).
		\end{flalign}
	\end{enumerate}
\end{theorem}
\begin{proof}
	See Appendix \ref{app:proofloweradd}.
\end{proof}
Theorem \ref{thm:lower} suggests that to meet the asymptotic universal lower bound on the WADD, an algorithm should be adaptive to the unknown $d_1$.

The proof of the asymptotic universal lower bound is based on a change-of-measure argument and the Weak Law of Large Numbers for log-likelihood ratio statistics, similarly to \cite{lai1998information}. However, a major difference is that when changing measures, the post-change statistic is more complicated, due to the cascading of the transient and persistent distributions. In the proof, a decomposition of the sum of the log-likelihood of the samples is necessary before the application of the Weak Law of Large Numbers.

\subsection{Asymptotic Upper Bounds on the WADD}
We now establish asymptotic upper bounds on the WADD of the proposed algorithms for a threshold $b$.  Again, we start with the case case of a single transient phase $(L=2)$, in which the 
WD-CuSum algorithm depends only on the parameter $\rho_1$.  Specifically, 
\begin{flalign}
	(\widetilde{W}[k])^+=\max\{\widetilde \mo^{(1)}[k],\widetilde \mo^{(2)}[k],0\},
\end{flalign}
where
\begin{flalign}\label{eq:51}
\widetilde \mo^{(1)}[k]&=\left(\widetilde \mo^{(1)}[k-1]\right)^+ + Z_1(X_k)+\log(1-\rho_1),\nn\\
\widetilde \mo^{(2)}[k]&=\max\left\{\log\rho_1,\widetilde \mo^{(1)}[k-1]+\log\rho_1, \widetilde \mo^{(2)}[k-1]\right\}\nn\\
&\quad+Z_2(X_k).
\end{flalign}
As we can observe from \eqref{eq:51}, the ``drift'' of the WD-CuSum algorithm for the samples within the transient phase is $I_1+\log(1-\rho_1)$, and there is a negative constant $\log\rho_1$ added to $\widetilde \mo^{(2)}[k]$.
To meet the asymptotic universal lower bound on the WADD, which does not depend on $\rho_1$, we need to mitigate the effect of $\rho_1$ on the performance.
If we choose $\rho_1$ such that  as $b\rightarrow \infty$,
\begin{flalign}\label{eq:rho}
\rho_1\rightarrow 0\text{ and }
\frac{\log \rho_1}{b}\rightarrow 0,
\end{flalign}
e.g., $\rho_1={1}/{b}$, then the ``effective drift" within the transient phase is $I_1(1-o(1))$, and the ``effective threshold" is $b(1+o(1))$, asymptotically. In this way, the effect of the weights  on the upper bound is asymptotically negligible.

Furthermore, suppose that there is a constant $c_1'\in[0,\infty]$, such that
\begin{flalign}\label{eq:db}
d_1\sim c_1'\frac{b}{I_1}.
\end{flalign}
If we choose $b\sim \log\gamma$, then $c_1=c_1'$, where $c_1$ is defined in \eqref{scaling}.

The following theorem characterizes asymptotic upper bounds on the WADD for the WD-CuSum and D-CuSum algorithms.
\begin{theorem}\label{thm:upper}
	Consider the QCD problem under transient dynamics described in Section \ref{sec:model} with $L=2$.
	Suppose that  \eqref{eq:rho} and \eqref{eq:db} hold. Consider the WD-CuSum algorithm in \eqref{eq:wdcusum}, and the D-CuSum algorithm in \eqref{eq:dcusum}. 
	
	(i) If $c_1'> 1$, then as $b\rightarrow\infty$,
	\begin{flalign}\label{eq:1}
	J^{d_1}_\text{L}(\widetilde \tau(b)) &=  J^{d_1}_\text{P}(\widetilde \tau(b)) \leq \frac{b}{I_1}(1+o(1)),\\
	J^{d_1}_\text{L}(\widehat \tau(b)) &=  J^{d_1}_\text{P}(\widehat \tau(b)) \leq \frac{b}{I_1}(1+o(1));
	\end{flalign}
	
	(ii) if $c_1'\leq 1$, then as $b\rightarrow\infty$,
	\begin{flalign}\label{eq:2}
	J^{d_1}_\text{L}(\widetilde \tau(b)) &=  J^{d_1}_\text{P}(\widetilde \tau(b))\leq b\left(\frac{c_1'}{I_1}+ \frac{1-c_1'}{I_2}\right)(1+o(1)),\\
		J^{d_1}_\text{L}(\widehat \tau(b)) &=  J^{d_1}_\text{P}(\widehat \tau(b))\leq b\left(\frac{c_1'}{I_1}+\frac{1-c_1'}{I_2}\right)(1+o(1)).
	\end{flalign}
\end{theorem}
\begin{proof}
By  Remark \ref{remark1} it follows that  for the D-CuSum the worse-case scenario for the observations up to the change-point $v_1$ is when $\widehat W[v_1]=0$, and consequently for every $b>0$ and $\bd$ we have
	\begin{flalign}
	J^\bd_\text{L}(	\widehat{\tau}(b)) &=  J^\bd_\text{P}(	\widehat{\tau}(b)) = \mE_1^\bd[	\widehat{\tau}(b)].
	\end{flalign}
Similarly we can argue that 
		\begin{flalign}
	J^\bd_\text{L}(	\widetilde{\tau}(b)) &=  J^\bd_\text{P}(	\widetilde{\tau}(b)) = \mE_1^\bd[	\widetilde{\tau}(b)].
	\end{flalign}
Thus,  the WADD for the D-CuSum and the WD-CuSum algorithms is achieved when $v_1=1$ under both Lorden's and Pollak's criteria.  Moreover,  by the construction of the D-CuSum and the WD-CuSum algorithms in \eqref{eq:dcusum} and \eqref{eq:wdcusum}, for any $k\geq 1$ we have
\begin{flalign}\label{eq:57}
\widetilde W[k]\leq \widehat W[k],
\end{flalign}
which is due to the fact that the weights in the WD-CuSum algorithm are less than one. Therefore,  with the same threshold $b$, the WD-CuSum algorithm will always stop later than the D-CuSum algorithm, thus
\begin{flalign} \label{eq:comp}
\mE_1^{d_1}[\widehat \tau(b)] \leq \mE_1^{d_1}[\widetilde \tau(b)],
\end{flalign}
and it  suffices to upper bound $\mE_1^{d_1}[\widetilde \tau(b)]$, which is done in Appendix \ref{app:upperbound}.
\end{proof}



The proof of the asymptotic upper bounds on WADD is based on an argument of partitioning the samples into independent blocks and the Law of Large Numbers for log-likelihood ratio statistics similar to those in \cite[Theorem 4]{lai1998information}. The major difficulty is due to the more complicated post-change statistic, which is a cascading of the transient and persistent distributions. In the proof, a novel approach of partitioning samples is needed to guarantee 
large probability of crossing the threshold within each block. Moreover, a decomposition of the sum of log-likelihood of the samples from $f_1$ and $f_2$, respectively, is also necessary before the application of the Law of Large Numbers.

The WADD is upper bounded differently in two regimes, depending on $c_1'$, which determines the scaling behavior between $d_1$ and $b$.
If $d_1$ is ``large", then the WD-CuSum algorithm stops within the transient phase with high probability, such that the asymptotic upper bound only depends on $I_1$; if $d_1$ is ``small", then the WD-CuSum algorithm stops within the persistent phase with high probability, such that the asymptotic upper bound depends on a mixture of $I_1$ and $I_2$. This is consistent with the insights gained from the asymptotic universal lower bound in Theorem \ref{thm:lower}.
%
%

\subsection{Asymptotic Optimality}
We are now ready to establish the asymptotic optimality of the proposed rules  with respect to both Lorden's and Pollak's criteria under every possible post-change regime.

\begin{theorem}\label{thm:optimal2}
	Consider the QCD problem under transient dynamics described in Section \ref{sec:model} with $L=2$. Suppose that    $d_1, \gamma \rightarrow \infty$ according to \eqref{scaling}.
	
	(i) If $b$ can be  selected so that $\widehat{\tau}(b) \in C_\gamma$ and 	 $b\sim\log\gamma$ as  	$\gamma \rightarrow \infty$, then 
	\begin{flalign}
	J^{d_1}_\text{L}( \widehat{\tau}(b)) &\sim   \inf_{\tau\in\mathcal C_\gamma}   J^{d_1}_\text{L}(\tau)
	\sim J^{d_1}_\text{P}(\widehat{\tau}(b)) \sim    \inf_{\tau\in\mathcal C_\gamma} J^{d_1}_\text{P}(\tau)\nn\\
	&\sim \left\{\begin{aligned}
	&\frac{\log\gamma}{I_1},&\text{ if }c_1>1,\\
	&\log\gamma\left(\frac{c_1}{I_1}+\frac{1-c_1}{I_2}\right), &\text{ if }c_1\leq 1.
	\end{aligned}\right. 
	\end{flalign}
	
	(ii) Suppose that $b$ is  selected so that 
	 $\mE_{\infty}[\widetilde\tau(b)] \geq \gamma$  and  $b\sim\log\gamma$ as  	$\gamma \rightarrow \infty$. If  $\rho_1 \rightarrow 0$ according to \eqref{eq:rho}, then
	\begin{flalign}
	J^{d_1}_\text{L}( \widetilde{\tau}(b)) &\sim    \inf_{\tau\in\mathcal C_\gamma} J^{d_1}_\text{L}(\tau)\sim
	J^{d_1}_P(\widetilde{\tau}(b)) \sim     \inf_{\tau\in\mathcal C_\gamma} J^{d_1}_\text{P}(\tau)\nn\\
	&\sim \left\{\begin{aligned}
	&\frac{\log\gamma}{I_1},&\text{ if }c_1>1,\\
	&\log\gamma\left(\frac{c_1}{I_1}+\frac{1-c_1}{I_2}\right), &\text{ if }c_1\leq 1.
		\end{aligned}\right. 
	\end{flalign}
\end{theorem}
\begin{proof}
	The results follow from Proposition \ref{prop:lowerarldcusum} and Theorems \ref{thm:lowerboundarl}, \ref{thm:lower}, and \ref{thm:upper}.
\end{proof}

\begin{figure}[htb]
	\centering
	\includegraphics[width=1\linewidth]{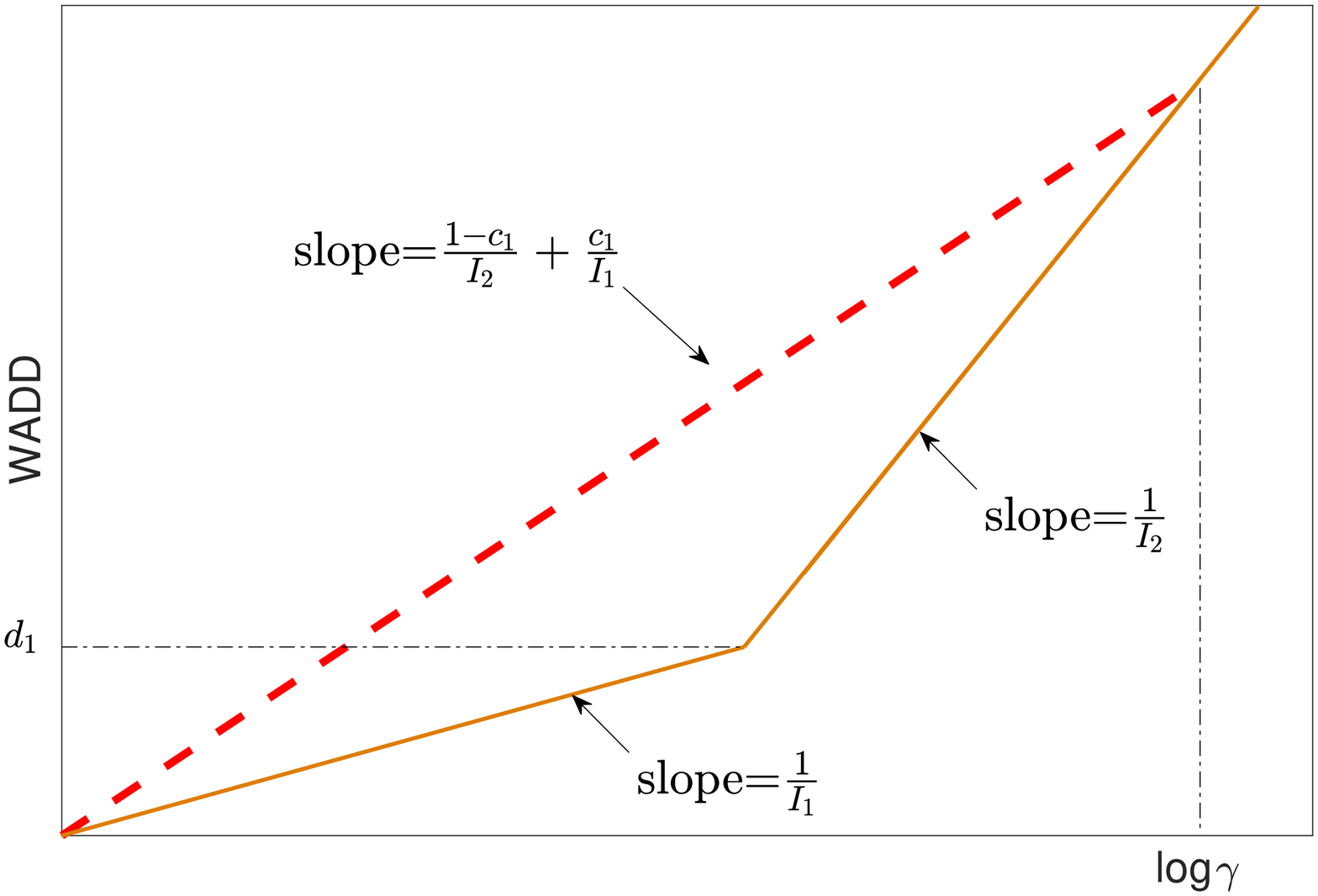}\\
	\caption{A heuristic explanation for the dichotomy in Theorem \ref{thm:optimal2}. }\label{fig:heuristic}
\end{figure}
A heuristic explanation for the dichotomy in Theorem \ref{thm:optimal2} is as follows (see also Fig.~\ref{fig:heuristic}). If we wish to detect a change from $f_0$ to $f_1$ with ARL $\gamma$, we have WADD$\sim{\log\gamma}/{I_1}$ (see, e.g., Theorem 1 in \cite{lai1998information}). However, we only have $d_1$ samples from $f_1$ within the transient phase. If $d_1\geq {\log\gamma}/{I_1}$, i.e., $c_1\geq 1$, then the problem is similar to one of testing the change from $f_0$ to $f_1$, and WADD increases when $\log\gamma$ increases with slope ${1}/{I_1}$, i.e., 
\begin{flalign}
\text{WADD}\sim\frac{\log\gamma}{I_1}.
\end{flalign}
 If $d_1<{\log\gamma}/{I_1}$, i.e., $c_1<1$, we then need further information from $f_2$, and WADD increases when $\log\gamma$ increases with slope ${1}/{I_2}$. To obtain the overall slope, it then follows that
\begin{flalign}
d_1I_1+(\text{WADD}-d_1)I_2&\approx \log\gamma,
\end{flalign}
which implies that
\begin{flalign}
\text{WADD}&\approx d_1+\frac{\log\gamma-d_1I_1}{I_2}\nn\\
&\sim\log\gamma\left(\frac{1-c_1}{I_2}+\frac{c_1}{I_1}\right).
\end{flalign}

\subsection{Generalization to Arbitrary $L$}
The asymptotic universal lower bound on the WADD can be extended to the case with arbitrary $L$.
\begin{theorem}\label{thm:lower_L}
	Consider the QCD problem under transient dynamics described in Section \ref{sec:model} with an arbitrary $L\geq 2$. Suppose that \eqref{scaling} holds.  If $h=\min\{1\leq j\leq L:\sum_{i=1}^{j}c_i\geq1\}$, then as $\gamma\rightarrow\infty$
	\begin{flalign}
	\inf_{\tau\in\mathcal C_\gamma} J^\bd_\text{L}(\tau)\nn &\geq \inf_{\tau\in\mathcal C_\gamma}J^\bd_\text{P}(\tau) \nn\\
	&\geq \log \gamma\left({\sum_{i=1}^{h-1}\frac{c_i}{I_i}}+\frac{1-\sum_{i=1}^{h-1}c_i}{I_h}\right)(1-o(1)).
	\end{flalign}
	
\end{theorem}
\begin{proof}
	The proof is a cumbersome but straightforward generalization of  {the proof of Theorem \ref{thm:lower}}, and is omitted.
\end{proof}

We further assume that there is a constant $c_i'\in[0,\infty]$ such that
\begin{flalign}\label{eq:dbL}
d_i\sim c_i'\frac{b}{I_1},
\end{flalign}
for $1\leq i\leq L-1$.
If we choose $b\sim \log\gamma$, then $c_i=c_i'$, $1\leq i\leq L-1$.

As in the case with $L=2$, we need to design the weights $\{\rho_i$, $1 \leq i \leq L-1\}$ in  the WD-CuSum algorithm so that their effect  is asymptotically negligible. Similarly to \eqref{eq:rho}, we choose $\rho_i$ such that as $b\rightarrow\infty$,
\begin{flalign}\label{eq:rho_L}
	\rho_i\rightarrow 0, \text{ and }\frac{-\log\rho_i}{b}\rightarrow 0,
\end{flalign}
for $i=1,\ldots, L-1$. We then obtain the following asymptotic upper bounds on the WADD of the D-CuSum and the WD-CuSum algorithms.

\begin{theorem}\label{thm:upper_L}
	Consider the QCD problem under transient dynamics described in Section \ref{sec:model} with an arbitrary $L$. Suppose   \eqref{eq:dbL} and \eqref{eq:rho_L}  hold.
	Let
	$h=\min\{1\leq j\leq L:\sum_{i=1}^{j}c'_i\geq1\}$, then as $\gamma\rightarrow\infty$
	\begin{flalign}
	J^\bd_\text{L}(\widehat \tau(b)) &=  J^\bd_\text{P}(\widehat \tau(b)) \leq J^\bd_\text{L}(\widetilde\tau(b)) =  J^\bd_\text{P}(\widetilde\tau(b)) \nn\\
	&\leq b\left({\sum_{i=1}^{h-1}\frac{c'_i}{I_i}}+\frac{1-\sum_{i=1}^{h-1}c'_i}{I_h}\right)(1+o(1)).
	\end{flalign}
\end{theorem}
\begin{proof}
	The proof is a cumbersome but straightforward generalization of the proof of Theorem \ref{thm:upper}, and is omitted.
\end{proof}

We are then ready to establish the asymptotic optimality of the proposed algorithms with respect to both Lorden's and Pollak's criteria under every possible post-change regime for  $L\geq 2$.
\begin{theorem}\label{thm:optimalL}
		Consider the QCD problem under transient dynamics described in Section \ref{sec:model} with $L\geq 2$. Assume that \eqref{scaling} is satisfied,
	as $\gamma, d \rightarrow \infty$. Let $h=\min\{1\leq j\leq L:\sum_{i=1}^{j}c_i\geq1\}$.
	
	(i) If $\exists b\sim\log\gamma$ such that $\mE_{\infty}[\widehat{\tau}(b))] \geq\gamma$. Then, as $\gamma\rightarrow \infty$,
	\begin{flalign}
	J^\bd_\text{L}( \widehat{\tau}(b)) &\sim\inf_{\tau\in\mathcal C_\gamma}   J^\bd_\text{L}(\tau)
	\sim J^\bd_\text{P}(\widehat{\tau}(b))\sim \inf_{\tau\in\mathcal C_\gamma} J^\bd_\text{P}(\tau)\nn\\
	&\sim    \log \gamma \left({\sum_{i=1}^{h-1}\frac{c_i}{I_i}}+\frac{1-\sum_{i=1}^{h-1}c_i}{I_h}\right).
	\end{flalign}
	
	(ii) Choose $\rho_i$ for $1\leq i\leq L-1$ such that \eqref{eq:rho_L} is satisfied and $b\sim\log\gamma$  such that  $\mE_{\infty}[\widetilde \tau(b)] \geq\gamma$.  Then,  as $\gamma\rightarrow \infty$,
	\begin{flalign*}
	J^\bd_\text{L}( \widetilde{\tau}(b)) &\sim \inf_{\tau\in\mathcal C_\gamma} J^\bd_\text{L}(\tau)\sim
	J^\bd_P(\widetilde{\tau}(b)) \sim \inf_{\tau\in\mathcal C_\gamma} J^\bd_\text{P}(\tau)\nn\\
	&\sim   \log \gamma\left({\sum_{i=1}^{h-1}\frac{c_i}{I_i}}+\frac{1-\sum_{i=1}^{h-1}c_i}{I_h}\right).
	\end{flalign*}
\end{theorem}
\begin{proof}
	The results follow from Proposition \ref{prop:lowerarldcusum} and Theorems \ref{thm:lowerboundarl}, \ref{thm:lower_L} and \ref{thm:upper_L}.
\end{proof}

A heuristic explanation for the {\em polychotomy} for  the general case with arbitrary $L$ in Theorem \ref{thm:optimalL} is as follows (see also Fig.~\ref{fig:heuristicL}). If we wish to test a change from $f_0$ to $f_1$ with ARL $\gamma$, we have WADD$\sim {\log\gamma}/{I_1}$. If $d_1<{\log \gamma}/{I_1}$, we further need samples from $f_2$. If $d_1I_1+d_2I_2$ is still less than $\log\gamma$, we then use samples from $f_3$. Up to the $h$-th transient phase, we have collected sufficient number of samples such that
\begin{flalign}
\sum_{i=1}^h d_iI_i>\log \gamma.
\end{flalign}
To obtain the overall slope, it then follows that
\begin{flalign}
\sum_{i=1}^{h-1} d_iI_i +\left(\text{WADD}-\sum_{i=1}^{h-1} d_i\right)I_h\approx \log\gamma,
\end{flalign}
which implies that
\begin{flalign}
\text{WADD}&\approx \frac{\log\gamma-\sum_{i=1}^{h-1} d_iI_i }{I_h}+\sum_{i=1}^{h-1} d_i\nn\\
&\sim \log\gamma\left(\frac{1-\sum_{i=1}^{h-1}c_i}{I_h}+\sum_{i=1}^{h-1}\frac{c_i}{I_i}\right).
\end{flalign}
\begin{figure}[htb]
	\centering
	\includegraphics[width=0.8\linewidth]{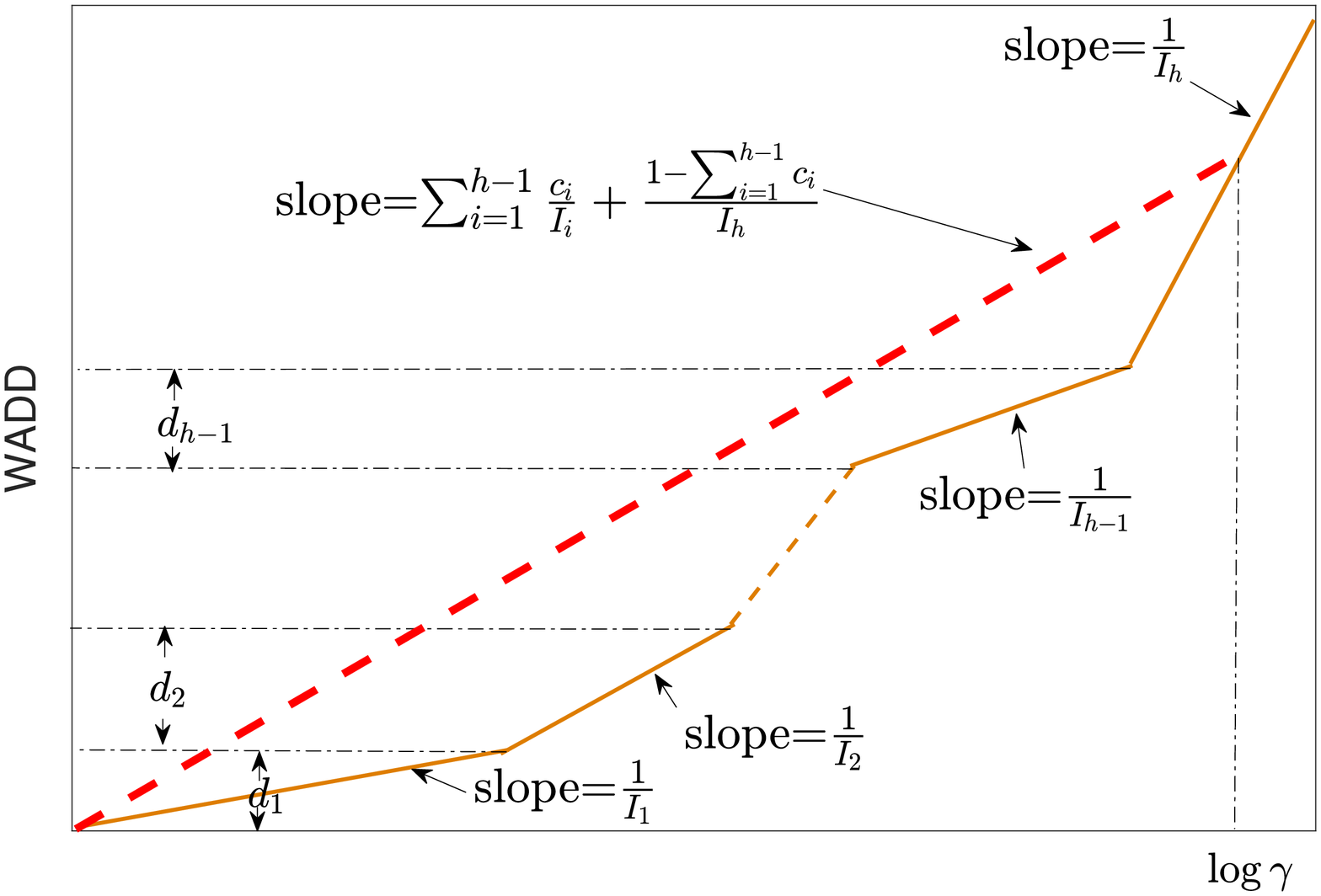}\\
	\caption{A heuristic explanation for the results of the general case with arbitrary $L$ in Theorem \ref{thm:optimalL}. }\label{fig:heuristicL}
\end{figure}

\section{Numerical Studies}\label{sec:numerical}
In this section, we present some numerical results. We focus on the case with $L=2$ to illustrate the performance of the algorithms and demonstrate our theoretical assertions.  Together with the insights gained from the theoretical results, we also propose a heuristic approach to assign  the weights for the WD-CuSum algorithm.

\begin{figure}[htb]
	\centering
	\includegraphics[width=0.8\linewidth]{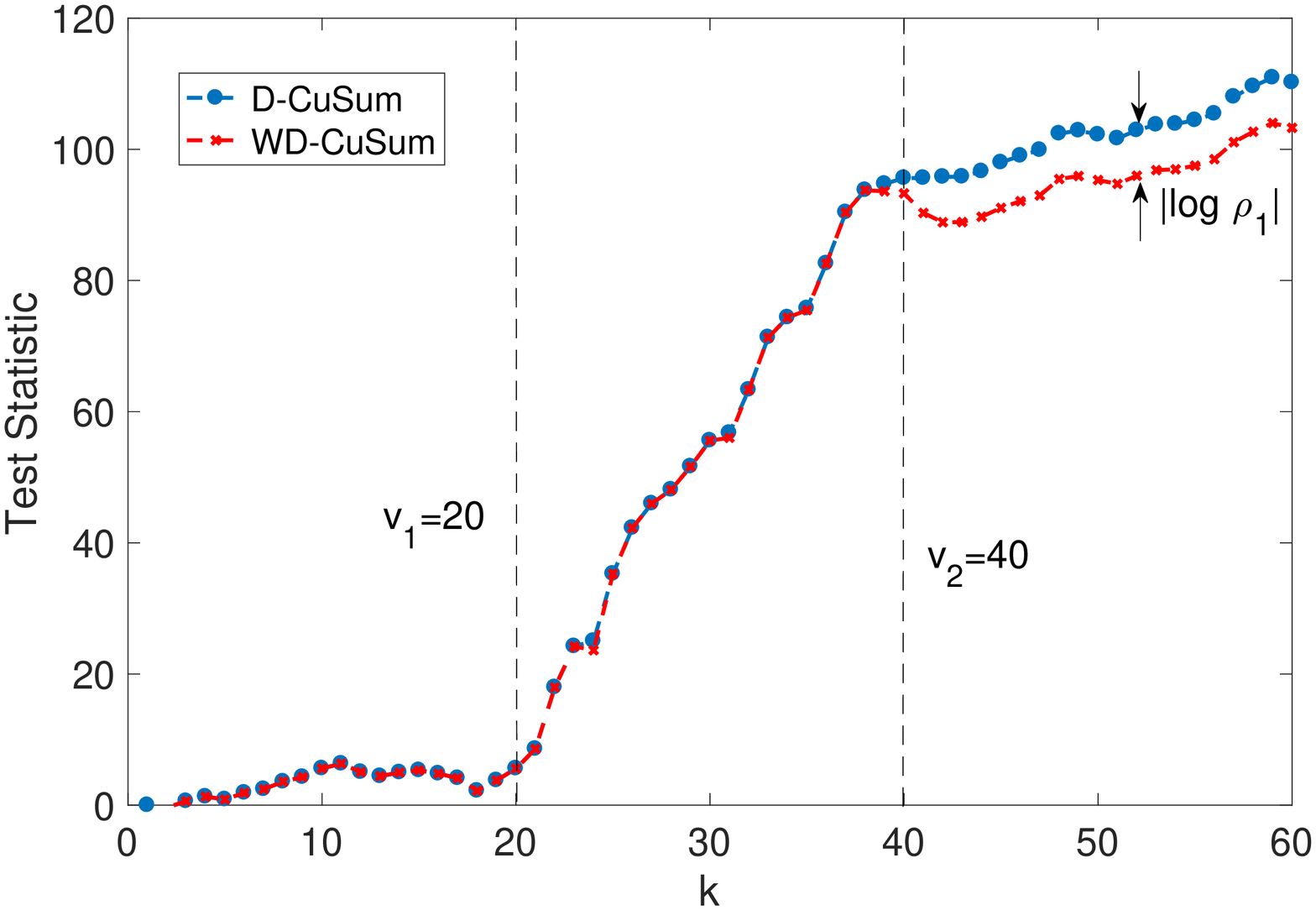}\\
	\caption{Evolution paths of the WD-CuSum and D-CuSum algorithms}\label{fig:evolutionpath}
\end{figure}
In Fig.~\ref{fig:evolutionpath}, we plot the evolution paths of the WD-CuSum and D-CuSum algorithms. We choose $f_0=\mathcal N(0,1)$, $f_1=\mathcal N(3,1)$ and $f_2=\mathcal N(1,1)$. We assume that the change happens at $v_1=20$ and the persistent phase starts at $v_2=40$. We choose $\rho_1={1}/{1000}$ for the WD-CuSum algorithm, which is small enough compared to $I_1$. It can be seen  that the values of both the WD-CuSum and D-CuSum algorithms stay close to zero before the change-point $v_1$ and grow after the change-point $v_1$ with different drifts in the transient and persistent phases. Both algorithms are seen to be adaptive to the unknown transient duration $d_1$. Furthermore, within the transient phase,  the WD-CuSum and D-CuSum algorithms have close evolution paths. After $v_2$, there is a gap of roughly $|\log\rho_1|$ between the two evolution paths. These observations reflect the difference between  the WD-CuSum and D-CuSum algorithms. For the D-CuSum algorithm, the drift is $I_1$ within the transient phase, and $I_2$ within the persistent phase. Recall that for the WD-CuSum algorithm, the drift within the transient phase is reduced from $I_1$ by $|\log(1-\rho_1)|$. Since $\rho_1$ is chosen to be small compared to $I_1$,  the change of drift is not significant in the figure. Furthermore, the value of the WD-CuSum statistic is reduced by $|\log\rho_1|$ within the persistent phase. Therefore, the difference between the values of the D-CuSum and the WD-CuSum statistics is roughly $|\log\rho_1|$ \gf{,} as shown in the figure.

\begin{figure}[htb]
	\centering
	\includegraphics[width=0.8\linewidth]{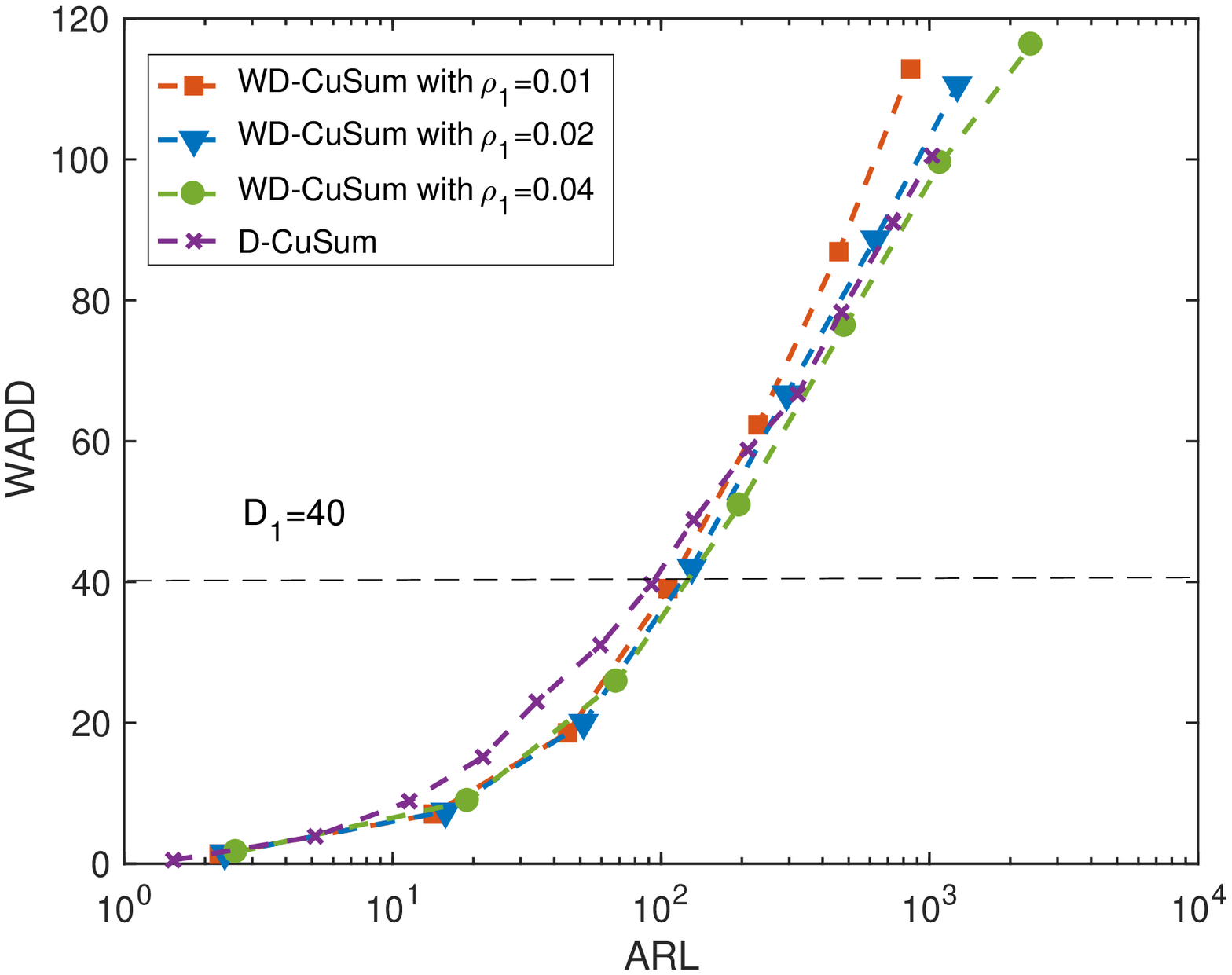}\\
	\caption{WADD versus ARL for the WD-CuSum and D-CuSum algorithms with $d_1=40$.}\label{fig:fig2}
\end{figure}

\begin{figure}[htb]
	\centering
	\includegraphics[width=0.8\linewidth]{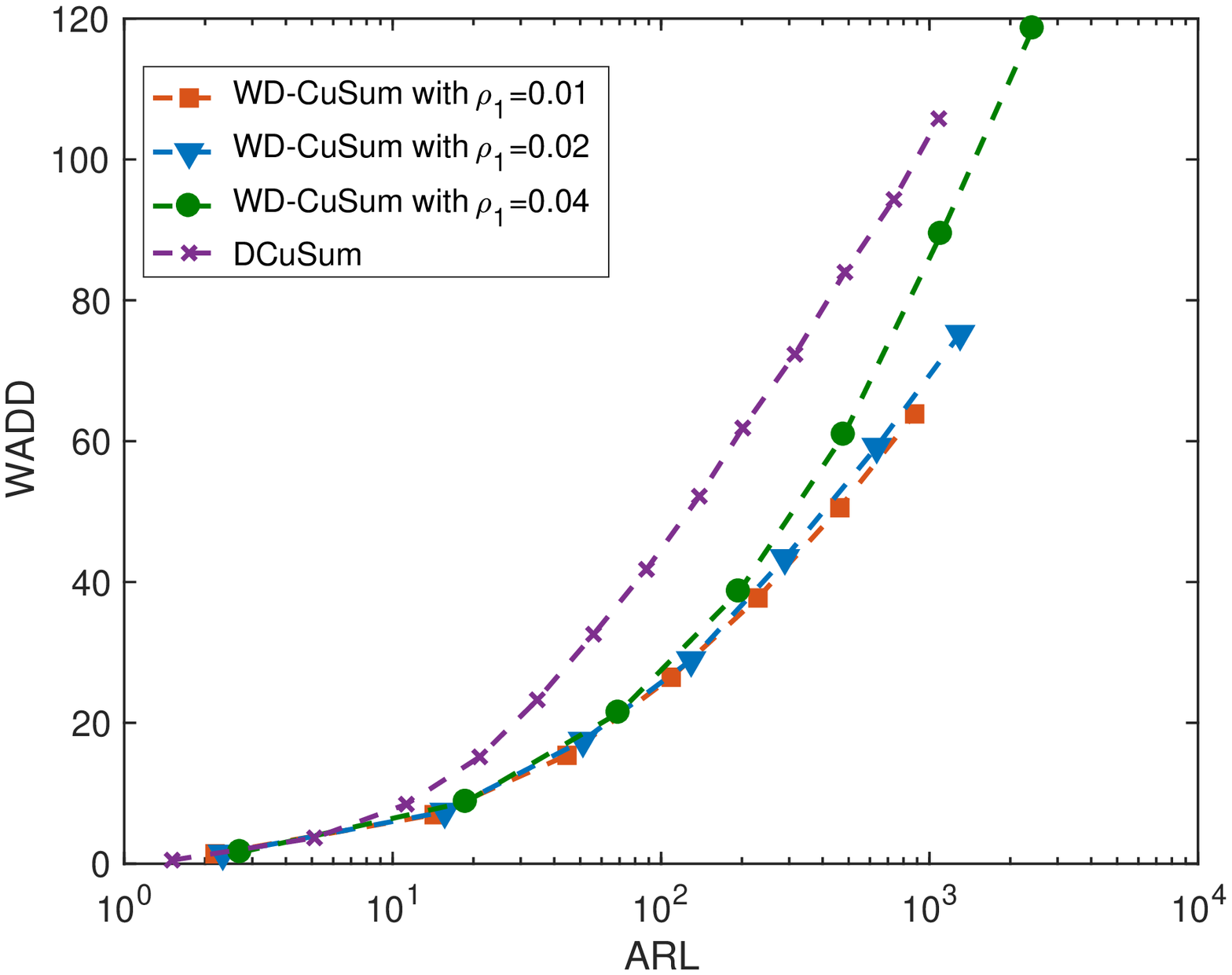}\\
	\caption{WADD versus ARL for the WD-CuSum and D-CuSum algorithms with $d_1=\infty$.}\label{fig:fig3}
\end{figure}

We next compare the performance of the WD-CuSum algorithms with different $\rho_1$ and the D-CuSum algorithm. The goal is to check how different choices of $\rho_1$ affect the performance of the WD-CuSum algorithm relative to the D-CuSum algorithm. We choose $f_0=\mathcal N(0,1)$, $f_1=\mathcal N(0.3,1)$ and $f_2=\mathcal N(-0.3,1)$. For the WD-CuSum algorithm, we consider three different choices of $\rho_1$, i.e., $\rho_1=0.01,0.02$ and $0.04$.
We choose $d_1=40$ and $d_1=\infty$, and plot the WADD versus the ARL in Fig.~\ref{fig:fig2} and Fig.~\ref{fig:fig3}, respectively. Fig.~\ref{fig:fig2} and Fig.~\ref{fig:fig3} show that if the algorithms stop within the transient phase, i.e., WADD$\leq d_1$, the WD-CuSum algorithm has a better performance than the D-CuSum algorithm. Fig.~\ref{fig:fig2} also shows that if the algorithms stop within the persistent phase, i.e., WADD$>d_1$, the D-CuSum and the WD-CuSum algorithms have similar performance.

In Fig.~\ref{fig:fig2}, when the algorithms stop within the persistent phase, i.e., WADD$>d_1$, the WD-CuSum algorithm has a better performance if $\rho_1$ is larger. This is due to the fact that the value of the WD-CuSum statistic is reduced by $|\log\rho_1|$ in the persistent phase, which slows down the detection. With a larger $\rho_1$, this effect is mitigated, which results in better performance for the WD-CuSum algorithm in the persistent phase.

In Fig.~\ref{fig:fig2} and more clearly in Fig~\ref{fig:fig3}, when the algorithms stop within the transient phase, i.e., WADD$\leq d_1$, the WD-CuSum algorithm has a better performance if $\rho_1$ is smaller. This is due to the fact that the drift of the WD-CuSum algorithm is reduced by $|\log(1-\rho_1)|$ in the transient phase, which also slows down the detection. With a smaller $\rho_1$, this effect is reduced, which results in better performance for the WD-CuSum algorithm in the transient phase.

As can be observed in Fig.~\ref{fig:fig2} and Fig.~\ref{fig:fig3}, the performance of the WD-CuSum algorithm depends on the choice of $\rho_1$, but not monotonically. A smaller $\rho_1$ yields a better performance for the WD-CuSum algorithm in the transient phase, and a larger $\rho_1$ yields a better performance for the WD-CuSum algorithm in the persistent phase. However, since $d_1$ is not known in advance, it is not clear in which regime the WD-CuSum algorithm will stop. Therefore, we propose a moderate way to choose $\rho_1$ that balances the performance within the transient  and  persistent phases.

Since the lower bound on the ARL in Theorem \ref{thm:lowerboundarl} does not depend on $\rho_1$, we choose $b\sim\log\gamma.$ We choose $\rho_1$ to be small but not too small such that the WD-CuSum algorithm is robust to the unknown $d_1$, i.e., the WD-CuSum algorithm has a good performance in both the transient and persistent phases.
Recall that the drift within the transient phase is reduced from $I_1$ by $|\log(1-\rho_1)|$. 
From our asymptotic analysis, we would like to have 
\begin{flalign}
\frac{-\log(1-\rho_1)}{I_1}\rightarrow 0, \text{ as } b\rightarrow\infty.
\end{flalign}
Therefore,  we let
\begin{flalign}
-\log(1-\rho_1)\leq \delta_1 I_1,
\end{flalign}
for some  $\delta_1\in(0,1)$,
such that the drift is reduced by a small fraction of $I_1$. Furthermore, within the persistent phase the value of the WD-CuSum statistic is reduced by $|\log\rho_1|$. 
From our asymptotic analysis, we would like to have
\begin{flalign}
\frac{-\log\rho_1}{b}\rightarrow 0, \text{ as } b\rightarrow\infty.
\end{flalign}
Therefore, we let 
\begin{flalign}
-\log\rho_1\leq \delta_2b,
\end{flalign}
for some $\delta_2\in(0,1)$, 
such that $|\log\rho_1|$ is a small perturbation compared to $b$. Therefore, $\rho_1$ is chosen such that 	
\begin{flalign}
e^{-\delta_2b}<\rho_1<1-e^{-\delta_1I_1}.
\end{flalign}
For example, we let $\delta_1=\delta_2=0.3$. Assume that $I_1=0.045$ (as in Fig.~\ref{fig:fig2} and Fig.~\ref{fig:fig3}) and the required ARL is $10^7$. Then we can choose $b=\log(10^7)$ and $\rho_1\in [0.008,0.134]$.

\section{Conclusions}\label{sec:con}
	In this paper, we studied a variant of the QCD problem that arises in a number of engineering applications. Our problem formulation captures the scenarios with transient dynamics after a change. We studied two algorithms for this formulation, the D-CuSum and the WD-CuSum algorithms. We established bounds on the ARL to false alarm for these algorithms that can be used to set the thresholds of these algorithms in application settings. We also established the asymptotic optimality of  the D-CuSum and the WD-CuSum algorithms up to a first-order asymptotic approximation. Both algorithms admit recursions that facilitate implementation and are adaptive to unknown transient dynamics. 

	We have shown that the asymptotic optimal performance follows a  polychotomy as illustrated in Fig.~\ref{fig:heuristicL}. In particular, for the case with only one transient phase, the asymptotic optimal performance follows a dichotomy: if the duration of the transient phase is ``large", then the WADD only depends on the distribution associated with the transient phase; otherwise, the WADD depends on the distributions associated with both the transient and the persistent phases.

	We note that in this paper, our asymptotic analysis is up to a first-order approximation. When the threshold or the transient durations are small, such an approximation may not be precise enough. It is therefore of  interest to develop more accurate approximations for the delay and false alarm rate of the algorithms.
%

	A possible extension of the problem formulation studied in this paper is a generalization to the case where the observations within each transient phase are not i.i.d.\gf{,} as in the observation model studied by Lai \cite{lai1998information}. Another extension is the scenario in which prior statistical knowledge of the change-point and durations of the transients is available. In this case, such prior knowledge should be incorporated into the design of algorithms to improve performance, while taking into account computational efficiency. We also note that the generalization to the case in which the distribution within each transient phase is composite is also of interest in practice, an example of which is the sequentially detection of a propagating event with an unknown propagation pattern in sensor networks \cite{zou2018icassp	}.

\appendix
\noindent {\Large \textbf{Appendix}}
\section{Proof of Proposition \ref{prop:lowerarldcusum}}\label{proof:proposition1}
Under \eqref{eq:conditions}, $\left\{(\widehat W[k])^+\right\}_{k\geq 1}$ is regenerative. Define the following regenerative times:
	\begin{flalign}
	{\sigma_1=\inf\left\{k: (\widehat{W}[k])^+ =0\right\}, (\inf \emptyset=\infty)}
	\end{flalign}
	{and}
	\begin{flalign}
	\sigma_{n+1}=\inf\left\{k>\sigma_n: (\widehat{W}[k])^+ =0\right\},
	\end{flalign}
	for $n\geq 1$. Let
	\begin{flalign}
	N=\inf\bigg\{ n\geq 0: \sigma_n\leq\infty \text{ and } (\widehat{W}[k])^+\geq b\nn\\
	\text{ for some } \sigma_n<k\leq \sigma_{n+1}   \bigg\}   
	\end{flalign}
	denote the index of the cycle in which $(\widehat{W}[k])^+$ crosses $b$.
	Then
	\begin{flalign}\label{eq:28}
	\mE_{\infty}[\hat{\tau}(b)]\geq\mE_{\infty}[N]=\sum_{n=0}^\infty \mP_{\infty}(N\geq n).
	\end{flalign}

For any $m\geq 1$,
\begin{flalign}
&	\mP_\infty(\widehat \tau(b)<Y)\nn\\
&=\mP_\infty(\widehat \tau(b)<Y, Y\leq m)+\mP_\infty(\widehat \tau(b)<Y, Y>m)\nn\\
&\leq \mP_\infty(\widehat \tau(b)< m)+\mP_\infty(Y>m)\nn\\
&\leq m^{L+1}e^{-b} +e^{-\alpha m},
\end{flalign}
where  the last inequality is due to condition \eqref{eq:conditions} and the following fact:
\begin{flalign}
&\mP_\infty(\widehat \tau(b)< m)\nn\\
&=\mP_\infty\left(\max_{1\leq k< m} \widehat W[k]>b\right)\nn\\
&=\mP_\infty\left(\max_{1\leq k < m} \max_{1\leq v_1\leq k}\max_{v_1\leq v_2\leq \cdots\leq v_L\leq k+1}\mg(k,v_1,\bd)>e^b\right)\nn\\
&\overset{(a)}{\leq} \sum_{1\leq k< m} \sum_{1\leq v_1\leq k}\sum_{v_1\leq v_2\leq \cdots\leq v_L\leq k+1}\mP_\infty\left(\mg(k,v_1,\bd)>e^b\right)\nn\\
&\overset{(b)}{\leq} m^{L+1}e^{-b},
\end{flalign}
and $(a)$ is due to the Boole's inequality \cite{durrett2010probability} and $(b)$ is due to Markov's inequality \cite{cover2012elements} and the fact that $\mE_{\infty}[\mg(k,v_1,\bd)]=1$.
By choosing $m={b/\alpha}$, it follows that
\begin{flalign}
&\mP_\infty(\widehat \tau(b)<Y)\leq {e^{-b}}{\left(\left(\frac{b}{\alpha}\right)^{L+1}+1\right)}.
\end{flalign}
Next, 
	\begin{flalign}\label{eq:301}
	&\mP_{\infty}(N\geq n)\nn\\
	&=\mP_{\infty}( (\widehat{W}[k])^+ <b, \forall k\leq \sigma_n)\nn\\
	&=\mP_{\infty}( (\widehat{W}[k])^+ <b, \forall \sigma_{m-1}\leq  k\leq \sigma_m, \forall 1\leq m\leq n)\nn\\
	&=\prod_{m=1}^{n}\mP_{\infty}( (\widehat{W}[k])^+ <b, \forall \sigma_{m-1}\leq  k\leq \sigma_m)\nn\\
	&\geq \left(1-{e^{-b}}{\left(\left(b/\alpha\right)^{L+1}+1\right)}\right)^{n},
	\end{flalign}
	where the last equality is due to the independence among the cycles \cite[Chapter 6.4]{asmussen2008applied}.
	Hence, combining \eqref{eq:28} and \eqref{eq:301}, it follows that
	\begin{flalign}
	\mE_{\infty}[\hat{\tau}(b)]&\geq \sum_{n=0}^\infty \left(1-{e^{-b}}{\left(\left(b/\alpha \right)^{L+1}+1\right)}\right)^{n}\nn\\
	&=\frac{e^{b}}{1+\left(b/\alpha \right)^{L+1}},
	\end{flalign}
	where the last step is due to the fact that for large $b$, ${e^{-b}}{\left(\left(b/\alpha \right)^{L+1}+1\right)}<1$. This concludes the proof.

\section{Proof of Theorem \ref{thm:lowerboundarl}}\label{proof:theorem1}
	For every $k \in \mathbb N$ we have
\begin{flalign}\label{eq:order}
&\widetilde{W}[k] \leq W'[k] \nn\\
&= \max_{1 \leq v_1 \leq k} \log\left(\sum_{\bd\in \mathbb N^{L-1}}\mg(k,v_1,\bd)g(\bd)\right)	\nn\\
&\leq  \log \left(\sum_{v_1=1}^{k} \sum_{\bd\in \mathbb N^{L-1}}\mg(k,v_1,\bd)g(\bd)\right)	\nn\\
&\equiv  \log R[k],
\end{flalign}
where $W'[k]$ is as in \eqref{eq:38}, and the first inequality follows by the construction of the detection statistics.  Note that  $R[k]$ is a mixture Shiryaev-Roberts statistics, and therefore $\{R[k]-k\}_{k\geq 1}$ is a martingale under $\mP_\infty$ \cite{pollak1987average}. Thus, for every $b>0$ and $k \in \mathbb N$ we have by Doob's submartingale inequality \cite{durrett2010probability} that
\begin{flalign}
&\mP_\infty( \widetilde{\tau}(b) \leq k) \nn\\
&=
\mP_\infty \left( \max_{1 \leq s \leq k} \widetilde{W}[s] \geq b \right) \nn\\
&\leq \mP_\infty \left(\max_{1 \leq s \leq k}   R[s] \geq e^b \right)\nn\\
&\leq k   e^{-b},
\end{flalign}
which implies  that
\begin{flalign}
\mE_{\infty}[\widetilde{\tau}(b)] &=\sum_{k=0}^{\infty} \mP_\infty( \widetilde{\tau}(b) >  k) \nn\\
&\geq  \sum_{k=0}^{\infty}
(1-k  e^{-b})^+ \nn\\
&=\sum_{k=0}^{e^b}\left(1-k   e^{-b}\right)\nn\\
&\geq \frac{e^b}{2}.
\end{flalign}

\section{A {Sufficient Condition} for \eqref{eq:conditions}}
\label{remark2} 
	Let
	\begin{flalign}\label{eq:phi}
	\mathrm \Phi(X_j)=\log\left( \frac{\max_{1\leq i\leq L}{f_i(X_j)}}{f_0(X_j)}\right).
	\end{flalign}
	%
	%
	For any $(v_1,\bd,k)$, it follows from \eqref{eq:gamma} and \eqref{eq:phi} that
	\begin{flalign}
	\log \mg(k,v_1,\bd)\leq \sum_{j=v_1}^k\mathrm \Phi(X_j).
	\end{flalign}
	This further implies that
	\begin{flalign}\label{eq:26}
	\widehat W[k]\leq \max_{1\leq v_1\leq k}\sum_{j=v_1}^k \mathrm \Phi(X_j).
	\end{flalign}
	Let $Y'=\inf\left\{k\geq 1: \max_{1\leq v_1\leq k}\sum_{j=v_1}^k \mathrm \Phi(X_j) \leq 0\right\}$. Then by \eqref{eq:26}, $Y' \geq Y$. It then follows that
	\begin{flalign}
	&\mP_{\infty}(Y>m)\nn\\
	&\leq \mP_{\infty}(Y'>m)\nn\\
	&=\mP_{\infty}\left(\max_{1\leq v_1\leq k}\sum_{j=v_1}^k \mathrm \Phi(X_j)> 0, \forall 1\leq k\leq m\right)\nn\\
	&\overset{(a)}{=}\mP_{\infty}\left(\sum_{j=1}^k \mathrm \Phi(X_j)> 0, \forall 1\leq k\leq m\right)\nn\\
	&\leq \mP_{\infty}\left(\sum_{j=1}^m \mathrm\Phi(X_j) > 0\right)\nn\\
	&=\mP_{\infty}\left(\sum_{j=1}^m \bigg(\mathrm \Phi(X_j) -\mE_{\infty}[\mathrm \Phi(X_j)]\bigg)> -m\mE_{\infty}[\mathrm \Phi(X_j)]\right),
	\end{flalign}
	where $(a)$ is by applying the following argument recursively:
	\begin{flalign}
	&\mP_\infty\Bigg(\mathrm \Phi(X_1)>0 \bigcap \bigg(\mathrm \Phi(X_2)>0 \bigcup \mathrm \Phi(X_1)+\mathrm \Phi(X_2)>0\bigg)\Bigg)\nn\\
	&=\mP_\infty\Bigg(\bigg(\mathrm \Phi(X_1)>0 \bigcap \mathrm \Phi(X_2)>0\bigg)\bigcup \nn\\
	&\hspace{0.2\linewidth} \bigg( \mathrm \Phi(X_1)>0\bigcap \mathrm \Phi(X_1)+\mathrm \Phi(X_2)>0\bigg) \Bigg)\nn\\
	&=\mP_\infty\bigg(\mathrm \Phi(X_1)>0\bigcap \mathrm \Phi(X_1)+\mathrm \Phi(X_2)>0 \bigg).
	\end{flalign}
	
	If 
		$
		\mE_{f_0}\left[\mathrm \Phi(X_j)\right]<0,$
		and $$-\alpha=\inf_{t>0} \big(\theta(t)+ t\mE_{f_0}	[\mathrm \Phi(X_j)]\big)<0,$$
		where \begin{flalign}
		\theta(t)=\log\mE_{f_0}\left[\exp\Big(t\big(	\mathrm \Phi(X_j)-\mE_{f_0}	[\mathrm \Phi(X_j)]\big)\Big)\right],
		\end{flalign}
		then by the Chernoff bound \cite{raginsky2013concentration}, \eqref{eq:conditions} holds.

\section{A Useful Lemma}
{We first recall the following  useful corollary of the Strong Law of Large Numbers, which will be used extensively.}
\begin{lemma}\cite[Lemma A.1]{fellouris2017multichannel}\label{lemma:slln}
	Suppose random variables $Y_1,Y_2,\ldots, Y_k$ are i.i.d. on $(\mathrm\Omega,\mathcal F,\mP)$ with $\mE[Y_i]=\mu>0$, and denote $S_k=\sum_{i=1}^kY_i$, then for any $\epsilon>0$, as $n\rightarrow \infty$,
	\begin{flalign}
	\mP\left(\frac{\max_{1\leq k\leq n}S_k}{n}-\mu>\epsilon\right)\rightarrow 0.
	\end{flalign}
\end{lemma}

\section{Proof of Theorem \ref{thm:lower}}\label{app:proofloweradd}
 Recall from \eqref{scaling} that $d_1\sim\frac{c_1\log\gamma}{I_1}$ for some $c_1\in[0,\infty]$. We define $K_\gamma$ as follows:
\begin{flalign}
K_\gamma=\left\{\begin{aligned}
&\frac{\log\gamma}{I_1}, &c_1&\in[1,\infty];\\
&\left(\frac{1-c_1}{I_2}+\frac{c_1}{I_1}\right)\log \gamma, &c_1&\in[0,1).
\end{aligned}\right.
\end{flalign}
Fix  any small enough $\epsilon>0$.
By Markov's inequality, we have
\begin{flalign}
\mE_{v_1}^{d_1}&[\tau-v_1|\tau\geq v_1]\nn\\
&\geq \mP_{v_1}^{d_1}\big(\tau-v_1\geq (1-\epsilon)K_\gamma|\tau\geq v_1\big)(1-\epsilon)K_\gamma.
\end{flalign}
It then suffices to show
\begin{flalign}\label{eq:goal1}
\sup_{\tau\in\mathcal C_\gamma}\mP_{v_1}^{d_1}(\tau-v_1< (1-\epsilon)K_\gamma|\tau \geq v_1)\rightarrow 0 \text{ as } \gamma\rightarrow \infty.
\end{flalign}
We will consider two cases depending on $c_1\geq 1$ or $c_1<1$.

\textbf{\emph{Case 1}}: Consider $c_1\geq 1$. Then $(1-\epsilon)K_\gamma < d_1$ for large $\gamma$.  We first have for every $a>0$,
\begin{small}
\begin{flalign}\label{eq:20}
&\mP_{v_1}^{d_1}(v_1\leq \tau<v_1+(1-\epsilon)K_{\gamma}| \tau\geq v_1)\nn\\
&=  \mP_{v_1}^{d_1}\left(v_1\leq \tau<v_1+(1-\epsilon)K_\gamma, \log\ml_1[v_1,\tau]\geq a\bigg| \tau\geq v_1\right)\nn\\
&\quad +\mP_{v_1}^{d_1}\left(v_1\leq \tau<v_1+(1-\epsilon)K_\gamma, \log\ml_1[v_1,\tau]< a\bigg| \tau\geq v_1\right)\nn\\
&\leq \mP_{v_1}^{d_1}\left( \max_{0\leq j< (1-\epsilon)K_\gamma}\log\ml_1[v_1,v_1+j]\geq a\bigg| \tau\geq v_1\right)\nn\\
&\quad +\mP_{v_1}^{d_1}\left(v_1\leq \tau<v_1+(1-\epsilon)K_\gamma, \ml_1[v_1,\tau]< e^a\bigg| \tau\geq v_1\right)\nn\\
&\overset{(a)}{=}\mP_{v_1}^{d_1}\left( \max_{0\leq j< (1-\epsilon)K_\gamma}\log\ml_1[v_1,v_1+j]\geq a\right)\nn\\
&\quad +\mP_{v_1}^{d_1}\left(v_1\leq \tau<v_1+(1-\epsilon)K_\gamma, \ml_1[v_1,\tau]< e^a\bigg| \tau\geq v_1\right),
\end{flalign}
\end{small}
where $(a)$ is due to the fact that $\log\ml_1[v_1,v_1+j]$ is independent of $X_1,\ldots,X_{v_1-1}$, $\forall 0\leq j< (1-\epsilon)K_\gamma$, and the fact that the event $\{\tau\geq v_1\}$ only depends on the random variables $X_1,X_2,\ldots,X_{v_1-1}$.

By changing the measure  $\mP^{d_1}_{v_1}$ to $\mP_\infty$ \cite[Proof of Theorem 7.1.3]{tartakovsky2014sequential}, it follows that
\begin{flalign}\label{eq:changemeasure1}
&\mP^{d_1}_{v_1}\left(v_1\leq \tau<v_1+(1-\epsilon)K_\gamma, \ml_1[v_1,\tau]\leq e^a\right)\nn\\
&\leq e^{a}\mE^{d_1}_{v_1}\left[\mathds{1}_{\{v_1\leq \tau<v_1+(1-\epsilon)K_\gamma, \ml_1[v_1,\tau]\leq e^a\}} \frac{1}{\ml_1[v_1,\tau]}\right]\nn\\
&\leq e^{a}\mE^{d_1}_{v_1}\left[\mathds{1}_{\{v_1\leq \tau<v_1+(1-\epsilon)K_\gamma\}} \frac{1}{\ml_1[v_1,\tau]}\right]\nn\\
&=e^{a}\mE_{\infty}\left[\mathds{1}_{\{v_1\leq \tau<v_1+(1-\epsilon)K_\gamma\}}\right]\nn\\
&=e^{a}\mP_{\infty}(v_1\leq \tau<v_1+(1-\epsilon)K_\gamma).
\end{flalign}
The event $\{\tau\geq v_1\}$ only depends on the random variables $X_1,X_2,\ldots,X_{v_1-1}$ that follow the same distribution $f_0$ under both $\mP_{\infty}$ and $\mP^{d_1}_{v_1}$. This implies  that
\begin{flalign}
\mP^{d_1}_{v_1}(\tau\geq v_1)=\mP_{\infty}(\tau\geq v_1).
\end{flalign}
It then follows from \eqref{eq:changemeasure1} that
\begin{flalign}\label{eq:19}
&  \mP^{d_1}_{v_1}\left(v_1\leq \tau<v_1+(1-\epsilon)K_\gamma,\ml_1[v_1,\tau]\leq e^a\big|\tau\geq v_1\right)\nn\\
&\leq e^{a}\mP_{\infty}(v_1\leq \tau<v_1+(1-\epsilon)K_\gamma|\tau\geq v_1).
\end{flalign}
Combining \eqref{eq:20} and \eqref{eq:19} yields that
\begin{flalign}\label{eq:27}
\mP&_{v_1}^{d_1}(v_1\leq \tau<v_1+(1-\epsilon)K_{\gamma}| \tau\geq v_1)\nn\\
&\leq e^{a} \mP_{\infty}(v_1\leq \tau<v_1+(1-\epsilon)K_\gamma| \tau\geq v_1)\nn\\
&\quad +\mP^{d_1}_{v_1}\left(\max_{0\leq j< (1-\epsilon)K_\gamma}\log\ml_1[v_1,v_1+j]\geq a\right).
\end{flalign}

Since $\mE_{\infty}[\tau]\geq \gamma$, then for each $m<\gamma$, there exists some $v_1\geq 1$, such that
\begin{flalign}\label{eq:condition1}
\mP_{\infty}(\tau\geq v_1)>0\text{ and }\mP_\infty (\tau<v_1+m|\tau\geq v_1)\leq \frac{m}{\gamma},
\end{flalign}
which can be shown by contradiction as in \cite[Theorem 1]{lai1998information}. Hence, for $m=(1-\epsilon)K_\gamma$, there exists $v_1$ such that
\begin{flalign}
\mP_{\infty}(v_1\leq \tau<v_1+(1-\epsilon)K_\gamma|\tau\geq v_1)\leq \frac{(1-\epsilon)K_\gamma}{\gamma}.
\end{flalign}
Set $a=(1-\epsilon^2)\log \gamma$, then
\begin{flalign}\label{eq:30}
&e^{a}\mP_{\infty}(v_1\leq \tau<v_1+(1-\epsilon)K_\gamma| \tau\geq v_1)\nn\\
&\leq \gamma^{1-\epsilon^2}\frac{(1-\epsilon)\log\gamma}{\gamma I_1}\rightarrow 0 \text{, as } \gamma\rightarrow\infty.
\end{flalign}

We next show that the second term in \eqref{eq:27} converges to zero as $\gamma\rightarrow\infty$. Because $c_1\geq 1$, for large $\gamma$, $d_1> (1-\epsilon)K_\gamma$, such that $X_j$, for $v_1\leq j< v_1+(1-\epsilon)K_\gamma$, are i.i.d. generated by $f_1$. Therefore, $Z_1(X_j)$, for $v_1\leq j< v_1+(1-\epsilon)K_\gamma$, are also i.i.d. with expectation $I_1$.
Rewrite $a=(1-\epsilon^2)\log \gamma=(1-\epsilon)K_\gamma I_1 (1+\epsilon)$, then
\begin{flalign}\label{eq:31}
&\mP^{d_1}_{v_1}\left(\max_{0\leq k< (1-\epsilon)K_\gamma}\log\ml_1[v_1,v_1+k]\geq a\right)\nn\\
&=
\mP_{v_1}^{d_1}\left(\max_{0\leq k< (1-\epsilon)K_\gamma} \sum_{j=v_1}^{v_1+k}Z_1(X_j)\geq a\right)\nn\\
&=\mP_{v_1}^{d_1}\left(\frac{\underset{0\leq k< (1-\epsilon)K_\gamma}{\max} \sum_{j=v_1}^{v_1+k}Z_1(X_j)}{(1-\epsilon)K_\gamma} -I_1\geq I_1\epsilon\right)\nn\\
&\rightarrow 0, \text{ as } \gamma\rightarrow \infty,
\end{flalign}
where the last step is by Lemma \ref{lemma:slln} and the fact that $Z_1(X_j)$, for $v_1\leq j <v_1+(1-\epsilon)K_\gamma$, are i.i.d. with expectation $I_1$.

Combining \eqref{eq:27}, \eqref{eq:30} and \eqref{eq:31} yields
\begin{flalign}
\mP^{d_1}_{v_1}(\tau-v_1<(1-\epsilon)K_\gamma|\tau\geq v_1)\rightarrow 0, \text{ as  }\gamma\rightarrow\infty.
\end{flalign}

\textbf{\emph{Case 2}}: Consider $c_1< 1$. 
Note that for any $c_1<1$, we have a small enough $\epsilon$ such that
\begin{flalign}
(1-\epsilon) \left(\frac{c_1}{I_1}+\frac{1-c_1}{I_2}\right)> \frac{c_1}{I_1}.
\end{flalign}
It then follows that $(1-\epsilon)K_\gamma-d_1\rightarrow \infty$ as $\gamma\rightarrow \infty$.

By a change-of-measure argument similar to case 1,
we obtain for any $a'>0$,
\begin{flalign}\label{eq:goal3}
\mP^{d_1}_{v_1}&(\tau<v_1+(1-\epsilon)K_\gamma|\tau\geq v_1)\nn\\
&\leq e^{a'}\mP_\infty(\tau<v_1+(1-\epsilon)K_\gamma|\tau\geq v_1)\nn\\
&\quad+\mP^{d_1}_{v_1}\left(\max_{0\leq j< (1-\epsilon)K_\gamma}   \log\mg(v_1+j,v_1,d_1) >a' \right).
\end{flalign}

Set 
\begin{flalign}\label{eq:a'}
a'=(1-\epsilon_1)\log \gamma,
\end{flalign}
 where $\epsilon_1=\frac{(1-c_1)\epsilon}{2}$, and let $m=(\frac{1-c_1}{I_2}+\frac{c_1}{I_1})(1-\epsilon)\log\gamma$ in \eqref{eq:condition1}.
Then, there exists $v_1$, such that as $\gamma\rightarrow \infty$
\begin{flalign}
e^{a'}\mP_\infty(T<v_1+(1-\epsilon)K_\gamma|T\geq v_1)\leq &\frac{(1-\epsilon)K_\gamma}{\gamma^{\epsilon_1}}\rightarrow 0. 
\end{flalign}

We next show that the second term in \eqref{eq:goal3} converges to zero as $\gamma\rightarrow \infty$.
It can be shown that
\begin{flalign}\label{eq:35}
&\max_{0\leq j< (1-\epsilon)K_\gamma}   \log\mg(v_1+j,v_1,d_1) \nn\\
&=\max_{0\leq j< (1-\epsilon)K_\gamma}\left(\sum_{k=v_1}^{v_1+\min\{d_1-1,j\}}Z_1(X_k)  +\sum_{k=v_1+d_1}^{v_1+j}Z_2(X_k)  \right)  \nn\\
&\leq \max_{0\leq j< (1-\epsilon)K_\gamma}\sum_{k=v_1}^{v_1+\min\{d_1-1,j\}}Z_1(X_k)  \nn\\
&\quad+ \max_{0\leq j< (1-\epsilon)K_\gamma}\sum_{k=v_1+d_1}^{v_1+j}Z_2(X_k)    \nn\\
%
&\overset{(a)}{=} \max_{0\leq j\leq d_1-1} \sum_{k=v_1}^{v_1+j}Z_1(X_k)   + 
\max_{d_1-1\leq j< (1-\epsilon)K_\gamma}\sum_{k=v_1+d_1}^{v_1+j}Z_2(X_k)\nn\\
&= \max_{0\leq j\leq d_1-1} \sum_{k=v_1}^{v_1+j}Z_1(X_k)   \nn\\
&\quad+ 
\max_{0\leq j< (1-\epsilon)K_\gamma-d_1+1}\sum_{k=v_1+d_1}^{v_1+d_1+j-1}Z_2(X_k),
\end{flalign}
where $(a)$ is due to the fact that if $j>d_1-1$, $\min\{d_1-1,j\}=d_1-1$, and  the fact that if $j<d_1$, $\sum_{k=v_1+d_1}^{v_1+j}Z_2(X_k)=0$.
By definition of $a'$ in \eqref{eq:a'},
\begin{flalign}
&a'=(1-\epsilon_1)\log\gamma \geq  E_1+E_2,
\end{flalign}
where 
\begin{flalign}
	E_1=&\left(1+\epsilon_1\right) c_1\log\gamma\nn\\
	\sim&(1+\epsilon_1)d_1I_1,\nn\\
	E_2=&\left( (1-\epsilon) \left(\frac{c_1}{I_1}+\frac{1-c_1}{I_2}\right)-\frac{c_1}{I_1}\right) I_2\log\gamma\left(1+\epsilon_1\right)\nn\\
	\sim&(1+\epsilon_1) \big( (1-\epsilon)K_\gamma-d_1\big)I_2.
\end{flalign}
Then,
\begin{flalign}\label{eq:36}
&\mP^{d_1}_{v_1}\left(\max_{0\leq j< (1-\epsilon)K_\gamma}   \log\mg(v_1+j,v_1,d_1) >a' \right)\nn\\
&\leq \mP_{v_1}^{d_1}\bigg(      \max_{1\leq j\leq d_1} \sum_{k=v_1}^{v_1+j-1}Z_1(X_k)   \nn\\
&\quad\quad\quad +  \max_{0\leq j< (1-\epsilon)K_\gamma-d_1+1}\sum_{k=v_1+d_1}^{v_1+d_1+j-1}Z_2(X_k)       >a' \bigg)\nn\\
&\leq \mP_{v_1}^{d_1}\bigg(      \max_{1\leq j\leq d_1} \sum_{k=v_1}^{v_1+j-1}Z_1(X_k)   \nn\\
&\quad\quad\quad +  \max_{0\leq j< (1-\epsilon)K_\gamma-d_1+1}\sum_{k=v_1+d_1}^{v_1+d_1+j-1}Z_2(X_k)       >E_1+E_2 \bigg)\nn\\
&\overset{(a)}{\leq} \mP_{v_1}^{d_1}\bigg(      \max_{1\leq j\leq d_1} \sum_{k=v_1}^{v_1+j-1}Z_1(X_k) >E_1\bigg)\nn\\
&\quad+ \mP_{v_1}^{d_1}\bigg( \max_{0\leq j< (1-\epsilon)K_\gamma-d_1+1}\sum_{k=v_1+d_1}^{v_1+d_1+j-1}Z_2(X_k)  > E_2\bigg)\nn\\
&\overset{(b)}{\rightarrow}0, \text{ as }\gamma\rightarrow\infty,
\end{flalign}
where $(a)$ is due to the fact that $\mP(Y_1+Y_2>y_1+y_2)\leq  \mP(Y_1>y_1) + \mP(Y_2>y_2)$ for any random variables $Y_1,Y_2$ and constants $y_1,y_2$,  and $(b)$ is due to Lemma \ref{lemma:slln}. This completes the proof.

\begin{figure*}[!h]
	\begin{flalign}\label{eq:wk}
	w[k_1,k_2,v_2]=\left\{\begin{aligned}
	&\log \frac{\left(\prod_{j=k_1}^{\min\{v_2-1,k_2\}}f_1(X_j)(1-\rho_1)\right)   \rho_1^{\mathds{1}_{\{k_2\geq v_2\}}} \prod_{j=v_2}^{k_2}f_2(X_j)}{\prod_{j=k_1}^{k_2}f_0(X_j)} , \text{ if } k_1\leq v_2,\\
	&\log \frac{\rho_1 \prod_{j=k_1}^{k_2}f_2(X_j)}{\prod_{j=k_1}^{k_2}f_0(X_j)} , \text{ if } k_1> v_2,\\
	\end{aligned}	
	\right.
	\end{flalign}
\end{figure*}

\section{Proof of Theorem \ref{thm:upper}}\label{app:upperbound}
We first show the asymptotic upper bound on the WADD for the WD-CuSum algorithm. Then the results for the D-CuSum algorithm naturally follows from \eqref{eq:comp}.

For notational convenience, define $w[k_1, k_2,v_2]$ as in \eqref{eq:wk},
i.e., $w[k_1,k_2,v_2]$ is the logarithm of the weighted likelihood ratio of the samples $X_{k_1},\ldots,X_{k_2}$ with the change-point $v_1=1$ and the starting point of the persistent phase being $v_2$.

We further note that the test statistic in \eqref{eq:wdcusum_test} is equivalent to
\begin{flalign}
(\widetilde W[k])^+=\max_{1\leq k_1\leq v_2\leq k+1} w[k_1,k,v_2].
\end{flalign}
Due to the Markov property and the recursive structure of $\{\widetilde \mo^{(1)}[k],\widetilde \mo^{(2)}[k]\}_{k\geq 1}$, it is clear that the WADD is achieved when $v_1=1$, i.e.,
\begin{flalign}\label{eq::3}
J^{d_1}_\text{L}(\widetilde \tau(b)) &=  J^{d_1}_\text{P}(\widetilde \tau(b)) = \mE_1^{d_1} [\widetilde \tau(b)].
\end{flalign}
It then suffices to upper bound  $\mE^{d_1}_1[\widetilde \tau(b)]$.	
When $\rho_1\rightarrow 0$ and $\frac{\log \rho_1}{b}\rightarrow 0$ as $b\rightarrow \infty$ and by the fact that $d_1\sim c_1'{b}/{I_1}$, we have
\begin{flalign}
{d_1}\sim c_1'{\frac{b}{I_1+\log(1-\rho_1)}}.
\end{flalign}

Depending on the value of $c_1'$, we bound $\mE^{d_1}_1[\widetilde \tau(b)]$ in the following two cases.

\textbf{\emph{Case 1:}} Consider $c_1'>1$. Our goal is to show that as $b\rightarrow \infty$,
\begin{flalign}
\mE^{d_1}_1[\widetilde \tau(b)]\leq \frac{b}{I_1}(1+o(1)).
\end{flalign}

In the following, we choose $\epsilon>0$ such that $1<\frac{1+\epsilon}{1-\epsilon}\leq c_1'$, i.e.,  $\frac{c_1'(1-\epsilon)}{1+\epsilon} \geq1$, and denote 
\begin{flalign}
n_b&=\frac{b(1+\epsilon)}{I_1+\log(1-\rho_1)},\\
c_\epsilon&=\left\lfloor c_1'\frac{(1-\epsilon)}{1+\epsilon}\right\rfloor.
\end{flalign}

We first have
\begin{flalign}\label{eq:109}
&\mE^{d_1}_1\left[\frac{\widetilde \tau(b)}{n_b}\right]\nn\\
&=\int_0^\infty \mP^{d_1}_1\left(\frac{\widetilde \tau(b)}{n_b}>x\right)dx\nn\\
&\leq \sum_{i=0}^\infty \mP^{d_1}_1\left({\widetilde \tau(b)}> {n_bi}\right)\nn\\
&=1+\sum_{i=1}^{c_\epsilon} \mP^{d_1}_1\left({\widetilde \tau(b)}> n_bi\right)+\sum_{i=c_\epsilon+1}^\infty \mP^{d_1}_1\left({\widetilde \tau(b)}> {n_bi}\right).
\end{flalign}
It then suffices to bound $\mP^{d_1}_1\left(\widetilde \tau(b)> n_bi\right)$ for the two regimes, $i\leq c_\epsilon$ and $i>  c_\epsilon$. We note that the event $\{{\widetilde \tau(b)}> {n_bi}\}$ only depends on the samples $X_1,\ldots,X_{n_bi}$. 

For $1\leq i\leq c_\epsilon$, $X_1,\ldots,X_{n_bi}$ are i.i.d. generated from $f_1$ under $\mP^{d_1}_1$.
Therefore,
\begin{flalign}\label{eq:41}
&\mP^{d_1}_1\left(\widetilde \tau(b)> n_bi\right)\nn\\
&=\mP^{d_1}_1\left( \max_{1\leq k\leq n_b i} (\widetilde W[k])^+\leq b\right)\nn\\
&=\mP^{d_1}_1\left(\max_{1\leq k\leq n_b i}    \hspace{0.1cm}   \max_{1\leq k_1\leq v_2\leq k+1} w[k_1,k,v_2] \leq b        \right)\nn\\
&\leq \mP^{d_1}_1\left(w[(u-1)n_b+1,u n_b,d_1+1]\leq b, \forall 1\leq u\leq i\right)\nn\\
&= \mP^{d_1}_1\left( \sum_{j=(u-1)n_b+1}^{u n_b}   \hspace{-0.06\linewidth}\big(Z_1(X_j)+\log(1-\rho_1)\big)\leq b, \forall 1\leq u\leq i\right)\nn\\
&\overset{(a)}{=}\prod_{u=1}^{i}\mP^{d_1}_1\left(\frac{1}{n_b} \sum_{j=(u-1)n_b+1}^{u n_b}  \hspace{-0.06\linewidth} \big(Z_1(X_j)+\log(1-\rho_1)\big)\leq \frac{b}{n_b}\right)\nn\\
&\overset{(b)}{\leq} \delta^i,
\end{flalign}
where $\delta$ can be arbitrarily small for large $b$, $(a)$ is due to the fact that $\{X_{1+(u-1)n_b},\ldots,X_{un_b}\}$ are independent from $\{X_{1+(u'-1)n_b},\ldots,X_{u'n_b}\}$ for any $u\neq u'$, and $(b)$ is by the Weak Law of Large Numbers.


For $i>c_\epsilon$, $n_bi>d_1$ for large $b$, then the samples $X_1,\ldots,X_{n_bi}$ are generated from different distributions, either $f_1$ or $f_2$.
We then define 
\begin{flalign}
t=\left\lceil \frac{I_1}{\min\{I_1,I_2\}}\right\rceil+1.
\end{flalign}
We note that $t$ is a constant that only depends on $I_1$ and $I_2$.


\begin{figure}[htb]
	\centering
	\includegraphics[width=0.6\linewidth]{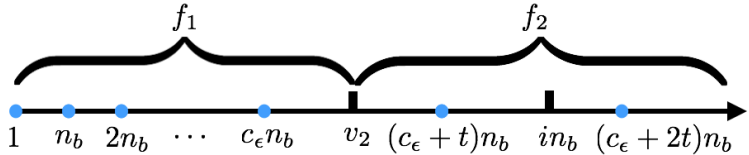}\\
	\caption{Illustration of partitioning of samples into blocks. Here $c_\epsilon +t\leq i <c_\epsilon+2t$. We partition the samples up to $in_b$ into $c_\epsilon$ blocks with size $n_b$, and one block with size $tn_b$. Then, the probability that the sum of log-likelihood of the samples within each block is less than $b$ is asymptotically small by the Weak Law of Large Numbers. The choice of block size $tn_b$ for the samples after $c_\epsilon n_b$ is due to the fact that the samples are generated from $f_1$ and $f_2$, and the need to guarantee an asymptotically small probability that the sum of log-likelihood of the samples within this block is less than $b$. }\label{fig:fig4}
\end{figure}

Consider any $i$ such that
$c_\epsilon+ (\ell-1) t   \leq i \leq c_\epsilon + \ell t -1$, for any $\ell\geq 1$, then
\begin{flalign}\label{eq:81}
\mP^{d_1}_1\left(\widetilde \tau(b)> n_bi\right)
&=\mP^{d_1}_1\left( \max_{1\leq k\leq n_b i} (\widetilde W[k])^+\leq b\right)\nn\\
&\leq \mP^{d_1}_1\left(A\cap B\right)\nn\\
&=\mP^{d_1}_1\left(A\right)\mP^{d_1}_1\left( B\right),
\end{flalign}
where
\begin{flalign}
A&=\left\{ w\left[1+(u-1)n_b,un_b,d_1+1\right] \leq b, \forall 1\leq u\leq c_\epsilon\right\},\\
B&=\big\{w\left[      \left( c_\epsilon + (u-1)t\right)  n_b+1,\left( c_\epsilon+ ut\right)  n_b           , d_1+1\right]  \leq b, \nn\\
&\quad\quad\quad \forall  1\leq u\leq \ell-1 \big\},
\end{flalign}
and the last equality is due to the fact that the events $A$ and $B$ are independent.  See Fig.~\ref{fig:fig4} for an illustration of partitioning the samples up to $n_bi$ into blocks with different sizes.

Similarly to \eqref{eq:41}, we obtain that
\begin{flalign}\label{eq:111}
\mP^{d_1}_1(A)\leq \delta ^ {c_\epsilon}.
\end{flalign}
Furthermore, by the Weak Law of Large Numbers, $\forall  1\leq u\leq \ell-1$, we have that as $b\rightarrow \infty$,
\begin{flalign}
\frac{w\left[      \left( c_\epsilon + (u-1)t\right)  n_b+1,\left( c_\epsilon+ ut\right)  n_b           , d_1+1\right] }{tn_b} \overset{p.}{\longrightarrow} I_2.
\end{flalign}
As $b\rightarrow \infty$,
\begin{flalign}
&\frac{w\left[      \left( c_\epsilon + (u-1)t\right)  n_b+1,\left( c_\epsilon+ ut\right)  n_b           , d_1+1\right] }{b} \nn\\
&\overset{p.}{\longrightarrow} \frac{I_2t(1+\epsilon)}{I_1}\geq 1+\epsilon.
\end{flalign}
Thus, 
\begin{flalign}
\mP^{d_1}_1 \left(w\left[      \left( c_\epsilon + (u-1)t\right)  n_b+1,\left( c_\epsilon+ ut\right)  n_b           , d_1+1\right]  \leq b\right)\leq \delta,
\end{flalign}
where $\delta$ can be arbitrarily small for large $b$. Then, it follows from similar arguments of independence  that
\begin{flalign}\label{eq:85}
\mP^{d_1}_1(B)\leq \delta^{ \ell-1}.
\end{flalign}

Combining \eqref{eq:111} and \eqref{eq:85} further implies that
\begin{flalign}\label{eq:117}
\mP^{d_1}_1&\left(\widetilde \tau(b)> n_bi\right)\leq \delta ^ {c_\epsilon+\ell-1}.
\end{flalign}

Hence, by \eqref{eq:109}, \eqref{eq:41} and \eqref{eq:117}, we have
\begin{flalign}
\mE^{d_1}_1\left[\frac{\widetilde \tau(b)}{n_b}\right]
&\leq \sum_{i=0}^{c_\epsilon} \delta^i  +\sum_{\ell=1}^{\infty}t\delta ^ {c_\epsilon+\ell-1}\nn\\
&=\frac{1}{1-\delta} + t\delta^{c_\epsilon} +(t-1)\delta^{c_\epsilon+1}\frac{1}{1-\delta}\nn\\
& \overset{\Delta}{=} 1+\delta',
\end{flalign}
where $\delta'$ can be arbitrarily small for large $b$ due to the facts that $c_\epsilon\geq 1$ and $\delta$ can be arbitrarily small for large $b$.
Therefore, as $b\rightarrow \infty$,
\begin{flalign}
\mE^{d_1}_1[\widetilde \tau(b)]\leq \frac{b}{I_1}(1+o(1)).
\end{flalign}

\textbf{\emph{Case 2:}} If $c_1' \leq 1$, our goal is to show that as $b\rightarrow \infty$,
\begin{flalign}
\mE^{d_1}_1[\widetilde \tau(b)]\leq b\left(\frac{c_1'}{I_1}+\frac{1-c_1'}{I_2}\right)(1+o(1)).
\end{flalign}


Let 
\begin{flalign}
n_b'&=\left(d_1+\frac{b-\log\rho_1-d_1(I_1+\log(1-\rho_1))}{I_2}\right)(1+\epsilon),\nn\\
&\sim b\left(\frac{c_1'}{I_1}+\frac{1-c_1'}{I_2}\right)(1+\epsilon).
\end{flalign} 
Then, we have
\begin{flalign}
\lim_{b\rightarrow\infty} \frac{n_b'}{d_1}=\left(1+\left(\frac{1}{c_1'}-1\right)\frac{I_1}{I_2}\right)(1+\epsilon)> 1,
\end{flalign}
which implies that for large $b$, $n_b'>d_1$, and $n_b'-d_1\rightarrow\infty$ as $b\rightarrow \infty$. 

To bound $\mE^{d_1}_1[\widetilde \tau(b)]$, we first obtain
\begin{flalign}
\mE^{d_1}_1\left[ \frac{\widetilde \tau(b)}{n_b'}\right]&\leq \sum_{i=0}^{\infty}\mP^{d_1}_1\left(\widetilde \tau(b) >  n_b' i\right)\nn\\
&=1+\sum_{i=1}^{\infty}\mP^{d_1}_1\left(\widetilde \tau(b) >  n_b' i\right).
\end{flalign}

If $i=1$,
\begin{flalign}\label{eq:94}
&\mP^{d_1}_1\left(\widetilde \tau(b) > n_b'\right)\nn\\
&=\mP^{d_1}_1\left(\max_{1\leq k\leq n_b'}(\widetilde W[k])^+\leq b\right)\nn\\
&\leq \mP^{d_1}_1\left( w\left[1,n_b',d_1+1\right]\leq b  \right)\nn\\
&=\mP^{d_1}_1\Bigg( \sum_{j=1}^{d_1}\bigg(Z_1(X_j)+\log(1-\rho_1)\bigg) +\log\rho_1\nn\\
&\quad\quad\quad +\sum_{j=d_1+1}^{n_b'}Z_2(X_j)  \leq b\Bigg)\nn\\
&=\mP^{d_1}_1\Bigg( \sum_{j=1}^{d_1}\bigg(Z_1(X_j)+\log(1-\rho_1)\bigg) +\sum_{j=d_1+1}^{n_b'}Z_2(X_j)  \nn\\
&\hspace{0.1\linewidth}\leq d_1(I_1+\log(1-\rho_1))   +  (n_b'-d_1)I_2 -\epsilon C\Bigg)\nn\\
&\overset{(a)}{\leq}\mP^{d_1}_1\Bigg(\sum_{j=1}^{d_1}\bigg(Z_1(X_j)+\log(1-\rho_1)\bigg)  \nn\\
&\quad\quad\quad\quad\quad\leq d_1(I_1+\log(1-\rho_1)) -\frac{\epsilon C}{2}\Bigg)\nn\\
&\quad+\mP^{d_1}_1\left(\sum_{j=d_1+1}^{n_b'}Z_2(X_j) \leq (n_b'-d_1)I_2 -\frac{\epsilon C}{2}\right)\nn\\
&\overset{(b)}{\leq} \delta,
\end{flalign}
where $C=d_1I_2+b-d_1(I_1+\log(1-\rho_1))-\log\rho_1$, $\delta$ can be arbitrarily small for large $b$, $(a)$ is due to the fact that for any random variables $X,Y$ and constants $x,y$, $\mP(X+Y\leq x+y)\leq \mP(X\leq x)+\mP(Y\leq y)$, and $(b)$ is due to the Weak Law of Large Numbers.

Define \begin{flalign}
t'=\left\lceil  \frac{1}{\left(\frac{c_1'}{I_1}+\frac{1-c_1'}{I_2}\right)\min\{I_1,I_2\}}\right\rceil+1,
\end{flalign}
 which only depends on $c_1'$, $I_1$ and $I_2$.
Following arguments similar to those in \eqref{eq:81}-\eqref{eq:85}, we can show that if $(\ell-1)t+1\leq i\leq \ell t$, for any $\ell\geq 1$,
\begin{flalign}\label{eq:95}
\mP^{d_1}_1\left(\widetilde \tau(b) > n_b'i\right)\leq t'\delta^{\ell}.
\end{flalign}
Combining \eqref{eq:94} with \eqref{eq:95} implies that
\begin{flalign}
\mE^{d_1}_1&\left[\frac{\widetilde \tau(b)}{n_b'}\right]\leq 1+\delta +\sum_{j=2}^\infty t'\delta^{j-1}\nn\\
&=\frac{1}{1-\delta}+t'\delta+(t'-1)\frac{\delta^2}{1-\delta}\nn\\
& \overset{\Delta}{=} 1+\delta'',
\end{flalign}
where $\delta''$ can be arbitrarily small for large $b$.
Therefore, as $b\rightarrow\infty$
\begin{flalign}
\mE^{d_1}_1[\widetilde \tau(b)]\leq b\left(\frac{c_1'}{I_1}+\frac{1-c_1'}{I_2}\right)(1+o(1)).
\end{flalign}

\bibliographystyle{IEEEbib}
\bibliography{QCD}

\end{document}